\documentclass[10pt,english]{amsart}
\usepackage{helvet}
\usepackage[T1]{fontenc}
\usepackage[latin1]{inputenc}
\usepackage{graphicx}
\advance\voffset by -24mm
\advance\hoffset by -20mm
\usepackage{pifont}
\usepackage{textcomp}
\usepackage{amssymb}

\makeatletter
\numberwithin{equation}{section} 
\numberwithin{figure}{section} 
 \theoremstyle{plain}
 \newtheorem{thm}{Theorem}
 \theoremstyle{definition}
 \newtheorem{defn}{Definition}[subsection]
 \theoremstyle{plain}
 \newtheorem{lem}[defn]{Lemma}
 \newtheorem{cor}[defn]{Corollary}
 \theoremstyle{remark}
 \newtheorem{rem}[defn]{Remark}
 \theoremstyle{definition}
 \newtheorem{example}[defn]{Example}
 \theoremstyle{plain}
 \newtheorem{prop}[defn]{Proposition}

\usepackage{latexsym,amsfonts, xypic,amssymb}
\usepackage{pstricks,pst-node,pst-tree}
\xyoption{all}

\newcommand{\p}{\mathbb{P}}
\newcommand{\pr}{\mathrm{pr}}
\renewcommand{\k}{\mathrm{k}}
\newcommand{\kk}{\mathrm{\overline{k}}}

\newcommand{\tr}{\!\triangleright\!}

\newcommand{\Pn}{\mathbb{P}^2}

\newcommand{\Aut}{\mathrm{Aut}}

\newcommand{\A}{\mathbb{A}}
\theoremstyle{remark}

\theoremstyle{definition}
\newtheorem{enavant}[defn]{}

\AtBeginDocument{

}

\usepackage{babel}
\makeatother

\author{J\'er\'emy Blanc} 
\address{J\'er\'emy Blanc, Universit\"{a}t Basel, Mathematisches Institut, Rheinsprung $21$, CH-$4051$ Basel, Switzerland.} 
\email{Jeremy.Blanc@unibas.ch} 
\author{Adrien Dubouloz} 
\address{Adrien Dubouloz, Institut de Math\'ematiques de Bourgogne, Universit\'e de Bourgogne, 9 avenue Alain Savary - BP 47870, 21078 Dijon cedex, France} 
\email{Adrien.Dubouloz@u-bourgogne.fr}

\thanks{First author supported by the Swiss National Science Foundation grant no PP00P2\_128422 /1. This research was supported in part by ANR Grant "BirPol" ANR-11-JS01-004-01 and PHC Grant Germaine de Sta\"el 26470XE.}

\begin{document}
\title{Affine surfaces with a huge group of automorphisms}

\begin{abstract}
We describe a family of rational affine surfaces $S$ with huge groups of automorphisms in the following sense: the normal subgroup $\mathrm{Aut}(S)_{\mathrm{alg}}$ of $\mathrm{Aut}(S)$ generated by all algebraic subgroups of $\mathrm{Aut}(S)$ is not generated by any countable family of such subgroups, and the quotient $\mathrm{Aut}(S)/\mathrm{Aut}(S)_{\mathrm{alg}}$ contains a free group over an uncountable set of generators.
\end{abstract}
\maketitle



\section*{Introduction}

The automorphism group of an algebraic curve defined of a field $k$
is always an algebraic group, of dimension at most $3$, the biggest
possible group being $\mathrm{Aut}(\mathbb{P}_{k}^{1})=\mathrm{PGL}(2,k)$. The situation is very different starting from dimension
$2$ even for complete or projective surfaces $S$: of course some
groups such as $\mathrm{Aut}(\mathbb{P}_{k}^{2})=\mathrm{PGL}(3,k)$
are still algebraic groups but in general $\mathrm{Aut}(S)$ only
exists as a group scheme locally of finite type over $k$ \cite{MaOo67}
and it may fail for instance to be algebraic group in the usual sense
because it has (countably) infinitely many connected components. This
happens for example for the automorphism group of the blow-up $S\to\mathbb{P}^{2}$
of the base-points of a general pencil of two cubics, which contains
a finite index group isomorphic to $\mathbb{Z}^{8}$, acting on the
pencil by translations. Note however that the existence of $\mathrm{Aut}(S)$
as a group scheme implies at least that $S$ has a largest connected
algebraic group of automorphisms: the identity component of $\mathrm{Aut}(S)$
equipped with its reduced structure. 

The picture tends to be much more complicated for non complete surfaces,
in particular affine ones. For instance the group $\mathrm{Aut}(\mathbb{A}_{k}^{2})$
of the affine plane $\mathbb{A}_{k}^{2}=\mathrm{Spec}(k[x,y])$ contains
algebraic groups of any dimension hence is very far from being algebraic.
In fact, the subgroups $$T_{n}=\left\{ (x,y)\mapsto(x,y+P(x)),P\in k[x],\,\deg P\leq n\right\} \simeq\mathbb{G}_{a,k}^{n+1}$$
of unipotent triangular automorphisms of degree at most $n$ form an increasing
family of connected subgroups of automorphism of $\mathbb{A}_{k}^{2}$
in the sense of \cite{Ram64} so that $\mathrm{Aut}(\mathbb{A}_{k}^{2})$
does not admit any largest connected algebraic group of automorphisms.
It is interesting to observe however that, as a consequence of Jung's
Theorem \cite{Jun}, $\mathrm{Aut}(\mathbb{A}_{k}^{2})$ is generated
by a countable family of connected algebraic subgroups, namely $\mathrm{GL}(2,k)$
and the above triangular subgroups $T_{n}$, $n\geq1$. A similar
phenomenon turns out to hold for other classical families of rational
affine surfaces with large groups of automorphisms: for instance
for the smooth affine quadric in $\mathbb{A}^{3}$ with equation
$xy-z^{2}+1=0$ and more generally for all normal affine surfaces
defined by an equation of the from $xy-P(z)=0$ where $P(z)$ is a
non constant polynomial, whose automorphism groups have been described
explicitly first by Makar-Limanov \cite{ML90} by purely algebraic methods
and more recently by the authors \cite{first} in terms of the birational
geometry of suitable projective models. 

All examples above share the common property that the normal subgroup
$\mathrm{Aut}(S)_{\mathrm{Alg}}\subset\mathrm{Aut}(S)$ generated
by all algebraic subgroups of $\mathrm{Aut}(S)$ is in fact generated
by a countable family of such subgroups and that the quotient $\mathrm{Aut}(S)/\mathrm{Aut}(S)_{\mathrm{Alg}}$
is countable. So one may wonder if such a property holds in general
for quasi-projective surfaces. It turns out that there exists
normal affine surfaces which have a much bigger group of automorphisms
and the purpose of this article is to describe explicitly one 
family of such surfaces. Our main result can be summarized as follows:
\begin{thm}\label{MainThm}
Let $k$ be an uncountable field and let $P,Q\in k[w]$ be polynomials
having at least $2$ distinct roots in an algebraic closure $\overline{k}$
ok $k$ and such that $P(0)\neq0$. Then for the affine surface $S$
in $\mathbb{A}^{4}=\mathrm{Spec}(k[x,y,u,v])$ defined by the system
of equations
\[
\begin{cases}
yu= & xP(x)\\
vx= & uQ(u)\\
yv= & P(x)Q(y)
\end{cases}
\]
the following holds:

1) The normal subgroup $\mathrm{Aut}(S)_{\mathrm{Alg}}\subset\mathrm{Aut}(S)$
is not generated by a countable union of algebraic groups,

2) The quotient $\mathrm{Aut}(S)/\mathrm{Aut}(S)_{\mathrm{Alg}}$
contains a free group over an uncountable set of generators. 
\end{thm}
Note that that the surfaces described in Theorem~\ref{MainThm} can be chosen
to be either singular or smooth, depending on the multiplicity of
the roots of $P$ and $Q$. 

The result is obtained from a systematic use of the methods developed
in \cite{first} for the study of affine surfaces admitting many $\mathbb{A}^{1}$-fibrations.
By virtue of pioneering work of Gizatullin \cite{Gi}, the latter
essentially coincide with surfaces admitting normal projective completions
$X$ for which the boundary divisor is a so-called zigzag, that is,
a chain $B$ of smooth proper rational curves supported in the smooth
locus of $X$. These have been extensively studied by Gizatullin and Danilov \cite{Gi,Gi-Da1,Gi-Da2} during the seventies and more recently by the authors
in \cite{first}. 

An important invariant of a zigzag is the sequence of self-intersections of its components, called its type, which in our context can be chosen to be of the form $(0,-1,-a_1,\dots,-a_r)$, where the $a_i\geq 2$ are a possibly empty sequence of integers. In this setting, the simplest possible zigzag has type $(0,-1)$ and the corresponding affine surface is the affine plane $\mathbb{A}^{2}$ viewed as the complement in the Hirzebruch surface $\rho_1:\mathbb{F}_1\rightarrow \mathbb{P}^1$ of the union of a fiber of $\rho_1$, with self-intersection $0$, and the exceptional section of $\rho_1$ with self-intersection $-1$. 
The next family in terms of the number of irreducible components in the boundary zigzag $B$ consists of types $(0,-1,-a_1)$, $a_1\geq 2$. These correspond to the normal hypersurfaces of $\A^3$ defined by equations of the form $xy-P(z)=0$ with $\deg(P)=a_1$ which were studied in detail in \cite{first}.

The present article is in fact devoted to the study of the next family, that is, affine surfaces corresponding to zigzags of types $(0,-1,-a_1,a_2)$, where $a_1,a_2\ge 2$. The surfaces displayed in Theorem~\ref{MainThm} provide explicit realisations of general members of this family as subvarieties of $\mathbb{A}^4$, but we describe more generally isomorphism classes and automorphism groups of all surfaces in the family. 

From this point view, Theorem~\ref{MainThm} above says that a relatively
minor increasing of the complexity of the boundary zigzag has very important
consequences on the geometry and the automorphism group of the inner
affine surface. The next cases, that is zigzags of type $(0,-1,-a_1,\dots,-a_r)$ with $r\ge 3$, could be studied in exactly the same way as we do here; the amount of work needed would just be bigger and more technical.\\

The article is organized as follows: in the first section we review
the main techniques introduced in \cite{first} to study normal affine
surfaces completable by a zigzag in terms of the birational geometry
of suitable projective models of them, called standard pairs. We also
characterize the nature of algebraic subgroups of their automorphism
groups in this framework (Proposition \ref{Prop:AlgGroups}). The
next two sections are devoted to the study of isomorphism classes of the
surfaces associated to zigzags of type $(0,-1,-a_1,-a_2)$ and the description of isomorphisms
between these in terms of elementary birational links between the
corresponding standard pairs. In the last section, we apply these
intermediate results to obtain a the structure of their
automorphism groups and of their $\A^1$-fibrations. Theorem~\ref{MainThm} is then a direct consequence of this precise description.

\section{Recollection on standard pairs and their birational geometry}

\subsection{Standard pairs and associated rational fibrations}
\indent\newline\noindent Recall that a \emph{zigzag} on a normal projective surface $X$ is
a connected SNC-divisor $B$, supported in the smooth locus of $X$, with irreducible components isomorphic to the projective line over $\k$ and whose dual graph is a chain. In what follows we always assume that the irreducible components  $B_i$, $i=0,\ldots,r$ of $B$ are ordered in such a way that \[
B_{i}\cdot B_{j}=\begin{cases}
1 & \textrm{if }\left|i-j\right|=1,\\
0 & \textrm{if }\left|i-j\right|>1.\end{cases}\]
and we write $B=B_0\tr B_1 \tr \cdots \tr B_r$ for such an oriented zigzag. The sequence of integers $\left((B_0)^2,\ldots,(B_r)^2\right)$ is then called the \emph{type} of $B$.

\begin{defn}\label{StandZigDef}
A \emph{standard pair}\footnote{these were called 1-standard in \cite{first}} is a pair $(X,B)$ consisting of a normal rational projective surface $X$ and an ordered zigzag $B$ that can be written as $B=F\tr C \tr E$ where $F$ and $C$ are smooth irreducible rational curves with self-intersections $F^2=0$ and
$C^2=-1$, and where $E=E_{1}\tr\cdots\tr E_{r}$ is a (possibly empty) chain of irreducible rational curves with self-intersections $(E_i)^2\leq-2$ for every $i=1,\ldots,r$. The \emph{type} of the the pair $(X,B)$ is the type $(0,-1,-a_1,\ldots ,-a_r)$ of its ordered zigzag $B$.    
\end{defn}

\begin{enavant} \label{AffineFiberedDef} The underlying projective surface of a standard pair $\left(X,B=F\tr C\tr E\right)$ comes equipped with a rational fibration $\bar{\pi}=\bar{\pi}_{\left|F\right|}:X\rightarrow\mathbb{P}^{1}$ defined by the complete linear system $\left|F\right|$. The latter restricts on the quasi-projective surface  
$S=X\setminus B$ to  a faithfully flat morphism $\pi:S\rightarrow \mathbb{A}^1$ with generic fiber isomorphic
to the affine line over the function field of $\mathbb{A}^{1}$, called an \emph{$\mathbb{A}^1$-fibration}. We use the notations $(X,B,\bar{\pi})$ and $(X\setminus B,\bar{\pi}\mid_{X\setminus B})$ (or simply, $(X\setminus B, \pi)$ when we consider the corresponding surfaces as equipped with these respective fibrations). 

When $B$ is the support of an ample divisor, $S$ is an affine surface and $\pi:S\rightarrow\mathbb{A}^{1}$ has a unique degenerate fiber $\pi^{-1}\left(\bar{\pi}\left(E\right)\right)$ constisting of a nonempty disjoint union of affine lines,  possibly defined over finite algebraic extensions of $\k$,  when equipped with its reduced
scheme structure. Furthermore, if any, the singularities of $S$ are all supported on the degenerate
fiber of $\pi$  and admit a minimal resolution whose exceptional set consists of a chain of rational curves possibly defined over a finite algebraic extension of $\k$. In particular, if $\k$ is algebraically closed of characteristic $0$, then $S$ has at worst Hirzebruch-Jung cyclic quotient singularities. Furthermore, according to \cite[Lemma 1.0.7]{first}, the minimal resolution of singularities of singularities $\mu:\left(Y,B\right)\rightarrow\left(X,B\right)$ of the pair $(X,B)$ can be obtained  from the first Hirzebruch surface $\rho\colon \mathbb{F}_1=\mathbb{P}(\mathcal{O}_{\p^1}\oplus \mathcal{O}_{\p^1}(1))\to \p^1$ by a uniquely determined sequence of blow-ups $\eta\colon Y\to \mathbb{F}^1$ restricting to isomorphisms outside 
the degenerate fibers of $\bar{\pi}\circ\mu$ in such a way that we have a commutative diagram
 \[\xymatrix@R=1mm@C=1cm{&& Y \ar[dll]_{\mu} \ar[dd]^{\mu\circ \bar{\pi}} \ar[drr]^{\eta} \\ X \ar[drr]_{\bar{\pi}}& & & &\mathbb{F}_1 \ar[dll]^{\rho} \\ && \mathbb{P}^1. }\]
\end{enavant}

\subsection{Birational maps of standard pairs}\label{links-recolec}
\indent\newline\noindent A birational map $\phi: \left(X,B\right)\dasharrow\left(X',B'\right)$ between standard pairs is a birational map $X\dasharrow X'$ which restricts to an  isomorphism $X\setminus B\stackrel{{\sim}}{\rightarrow} X'\setminus B'$. It is an isomorphism of pairs if it is moreover an isomorphism from $X$ to $X'$. The birational maps between standard pairs
play a central role in the study of the automorphism groups of $\mathbb{A}^1$-fibered affine surfaces as in \ref{AffineFiberedDef} above. The main result of \cite{first} asserts the existence of a decomposition of every such birational map into a finite sequence of "basic" birational maps of standard pairs called \emph{fibered modifications} and \emph{reversions} which can be defined respectively as follows:

\begin{defn} \label{FiberedModif} (\cite[Definition 2.2.1 and Lemma 2.2.3]{first}). A \emph{fibered modification} is a strictly birational map of standard pairs  $$\phi: (X,B=F\tr C \tr E) \dashrightarrow (X',B'=F' \tr C' \tr E')$$ which induces an isomorphism of $\mathbb{A}^1$-fibered quasi-projective surfaces \[\xymatrix@R=0.4cm@C=1cm{S=X\setminus B \ar[d]_{\bar{\pi}\mid_{S}} \ar[r]^-{\sim}_-{\phi} & S'=X'\setminus B' \ar[d]^{\bar{\pi}'\mid_{S'}} \\ \mathbb{A}^1 \ar[r]^{\sim} & \mathbb{A}^1}\] where $\bar{\pi}\mid_S$ and $\bar{\pi}'\mid_S'$ denotes the restrictions of the rational pencils defined by the complete linear systems $\left|F\right|$ and $\left|F'\right|$ on $X$ and $X'$ respectively. Equivalently, with the notation of \S \ref{AffineFiberedDef}, the birational map $\left(\mu'\right)^{-1}\circ\phi\circ\mu:Y\dashrightarrow Y'$ induced by $\phi$ is the lift via $\eta$ and $\eta'$ of a non affine isomorphism of $\mathbb{A}^{1}$-fibered
affine surfaces \[\xymatrix@R=0.4cm@C=1cm{ \mathbb{A}^2=\mathbb{F}_1\setminus (\eta(F)\cup\eta(C)) \ar[d]_{\rho\mid_{\mathbb{A}^2}} \ar[r]_-{\Psi}^-{\sim} & \mathbb{A}^2=\mathbb{F}_1\setminus (\eta'(F')\cup\eta'(C')) \ar[d]^{\rho\mid_{\mathbb{A}^2}} \\ \mathbb{A}^1 \ar[r]_{\psi}^{\sim} & \mathbb{A}^1 }\] which maps isomorphically the base-points of $\eta^{-1}$  onto those of $\eta'^{-1}$. 
\end{defn}

\begin{defn} \label{DescriptionReversion1} (\cite[\S 2.3]{first}). A \emph{reversion} is a very special kind of birational map of standard pairs uniquely determined  by the choice of a $\k$-rational point $p\in F\setminus E$ and obtained by the following construction :

 Starting from a pair $(X,B=F\tr C\tr E)$ of type $(0,-1,-a_1,\ldots , -a_r)$, the contraction of the $(-1)$-curve $E$ followed by the blow-up of $p\in F\setminus E$ yields a birational map $\theta_0:(X,B)\dasharrow (X_0,B_0)$ to a pair with a zigzag of type $(-1,0, -a_1+1,\ldots ,-a_r)$. Preserving the fibration given by the $(0)$-curve, one can then construct an unique birational map $\varphi_1:(X_0,B_0)\dasharrow (X_1',B_1')$, where $B_1'$ is a zigzag of type  $(-a_1+1,0,-1,-a_2,\ldots, -a_r)$. The blow-down of the $(-1)$-curve followed by the blow-up of the point of intersection of the $(0)$-curve with the curve immediately after it yields a birational map $\theta_1:(X_1',B_1')\dasharrow (X_1,B_1)$ where $B_1$ is a zigzag of type $(-a_1,-1,0,-a_2+1,\ldots,-a_r)$. Repeating this procedure eventually yields birational maps $\theta_0,\varphi_1,\theta_1,...,\varphi_r,\theta_r$ described by the following figure.

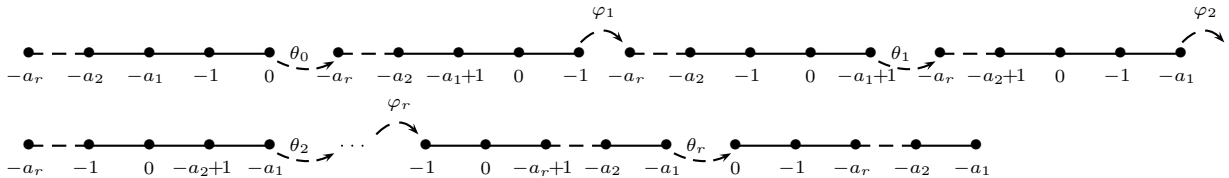
\begin{center}\begin{tabular}{l}
\begin{pspicture}(0,0.4)(4,1.8)
\psline(0.8,1)(3.2,1)
\psline[linestyle=dashed](0,1)(0.8,1)
\rput(0,1){\textbullet}\rput(-0.05,0.7){{\scriptsize $-a_r$}}
\rput(0.8,1){\textbullet}\rput(0.75,0.7){{\scriptsize $-a_2$}}
\rput(1.6,1){\textbullet}\rput(1.55,0.7){{\scriptsize $-a_1$}}
\rput(2.4,1){\textbullet}\rput(2.35,0.7){{\scriptsize $-1$}}
\rput(3.2,1){\textbullet}\rput(3.2,0.7){{\scriptsize $0$}}
\rput(3.6,1){{\scriptsize $\theta_0$}}
\parabola[linestyle=dashed]{->}(3.3,0.9)(3.68,0.8)
\end{pspicture}
\begin{pspicture}(0,0.4)(3.76,1.5)
\psline(0.8,1)(3.2,1)
\psline[linestyle=dashed](0,1)(0.8,1)
\rput(0,1){\textbullet}\rput(-0.05,0.7){{\scriptsize $-a_r$}}
\rput(0.8,1){\textbullet}\rput(0.75,0.7){{\scriptsize $-a_2$}}
\rput(1.6,1){\textbullet}\rput(1.55,0.7){{\scriptsize $-a_1\!\!+\!\!1$}}
\rput(2.4,1){\textbullet}\rput(2.4,0.7){{\scriptsize $0$}}
\rput(3.2,1){\textbullet}\rput(3.15,0.7){{\scriptsize $-1$}}
\rput(3.53,1.55){{\scriptsize $\varphi_1$}}
\parabola[linestyle=dashed]{->}(3.22,1.15)(3.5,1.3)
\end{pspicture}
\begin{pspicture}(0,0.4)(4,1.5)
\psline(0.8,1)(3.2,1)
\psline[linestyle=dashed](0,1)(0.8,1)
\rput(0,1){\textbullet}\rput(-0.05,0.7){{\scriptsize $-a_r$}}
\rput(0.8,1){\textbullet}\rput(0.75,0.7){{\scriptsize $-a_2$}}
\rput(1.6,1){\textbullet}\rput(1.55,0.7){{\scriptsize $-1$}}
\rput(2.4,1){\textbullet}\rput(2.4,0.7){{\scriptsize $0$}}
\rput(3.2,1){\textbullet}\rput(3.15,0.7){{\scriptsize $-a_1\!\!+\!\! 1$}}
\rput(3.6,1){{\scriptsize $\theta_1$}}\parabola[linestyle=dashed]{->}(3.3,0.9)(3.68,0.8)\end{pspicture}
\begin{pspicture}(0,0.4)(3.76,1.5)
\psline(0.8,1)(3.2,1)
\psline[linestyle=dashed](0,1)(0.8,1)
\rput(0,1){\textbullet}\rput(-0.05,0.7){{\scriptsize $-a_r$}}
\rput(0.8,1){\textbullet}\rput(0.75,0.7){{\scriptsize $-a_2\!\!+\!\! 1$}}
\rput(1.6,1){\textbullet}\rput(1.6,0.7){{\scriptsize $0$}}
\rput(2.4,1){\textbullet}\rput(2.35,0.7){{\scriptsize $-1$}}
\rput(3.2,1){\textbullet}\rput(3.15,0.7){{\scriptsize $-a_1$}}
\rput(3.53,1.55){{\scriptsize $\varphi_2$}}
\parabola[linestyle=dashed]{->}(3.22,1.15)(3.5,1.3)\end{pspicture}
\\
\begin{pspicture}(0,0.4)(4,1.5)
\psline(0.8,1)(3.2,1)
\psline[linestyle=dashed](0,1)(0.8,1)
\rput(0,1){\textbullet}\rput(-0.05,0.7){{\scriptsize $-a_r$}}
\rput(0.8,1){\textbullet}\rput(0.75,0.7){{\scriptsize $-1$}}
\rput(1.6,1){\textbullet}\rput(1.6,0.7){{\scriptsize $0$}}
\rput(2.4,1){\textbullet}\rput(2.3,0.7){{\scriptsize $-a_2\!\!+\!\! 1$}}
\rput(3.2,1){\textbullet}\rput(3.15,0.7){{\scriptsize $-a_1$}}
\rput(3.6,1){{\scriptsize $\theta_2$}}\parabola[linestyle=dashed]{->}(3.3,0.9)(3.68,0.8)\end{pspicture}\begin{pspicture}(2.6,0.4)(3.76,1.5)
\psline[linestyle=dotted](2.75,1)(3.05,1)
\parabola[linestyle=dashed]{->}(3.22,1.15)(3.5,1.3)
\rput(3.53,1.55){{\scriptsize $\varphi_r$}}
\end{pspicture}
\begin{pspicture}(0,0.4)(4,1.5)
\psline(0,1)(1.6,1)
\psline(2.4,1)(3.2,1)
\psline[linestyle=dashed](1.6,1)(2.4,1)
\rput(0,1){\textbullet}\rput(-0.05,0.7){{\scriptsize $-1$}}
\rput(0.8,1){\textbullet}\rput(0.8,0.7){{\scriptsize $0$}}
\rput(1.6,1){\textbullet}\rput(1.55,0.7){{\scriptsize $-a_r\!\!+\!\! 1$}}
\rput(2.4,1){\textbullet}\rput(2.35,0.7){{\scriptsize $-a_2$}}
\rput(3.2,1){\textbullet}\rput(3.15,0.7){{\scriptsize $-a_1$}}
\rput(3.6,1){{\scriptsize $\theta_r$}}\parabola[linestyle=dashed]{->}(3.3,0.9)(3.68,0.8)\end{pspicture}
\begin{pspicture}(0,0.4)(3.5,1.5)
\psline(0,1)(1.6,1)
\psline(2.4,1)(3.2,1)
\psline[linestyle=dashed](1.6,1)(2.4,1)
\rput(0,1){\textbullet}\rput(0,0.7){{\scriptsize $0$}}
\rput(0.8,1){\textbullet}\rput(0.75,0.7){{\scriptsize $-1$}}
\rput(1.6,1){\textbullet}\rput(1.55,0.7){{\scriptsize $-a_r$}}
\rput(2.4,1){\textbullet}\rput(2.35,0.7){{\scriptsize $-a_2$}}
\rput(3.2,1){\textbullet}\rput(3.15,0.7){{\scriptsize $-a_1$}}
\end{pspicture}
\end{tabular}\end{center}

\noindent The \emph{reversion of $(X,B)$ with center at $p$} is then the strictly birational map of standard pairs  $$\phi=\theta_r\varphi_r\cdots\theta_1\varphi_0\theta_0:(X,B)\dasharrow(X_r,B_r)=(X',B').$$ Note that the above construction is symmetric so that the inverse $\phi^{-1}:(X',B')\dashrightarrow (X,B)$ of $\phi$ is again a reversion, with center at its unique proper base point $p'=\phi(E)\in F'\setminus E'$.
\end{defn}

With these definitions, the decomposition results established in \cite[Theorem 3.0.2 and Lemma 3.2.4]{first} can be summarized as follows:  

\begin{prop}\label{mainfirst}
 For a birational map of standard pairs $\phi: \left(X,B\right)\dasharrow\left(X',B'\right)$, the following holds
\begin{enumerate}
\item
 The map $\phi$ is either an isomorphism of pairs or it can decomposed into a finite sequence \[\phi=\phi_{n}\circ\cdots\circ\phi_{1}:\left(X,B\right)=\left(X_{0},B_{0}\right)\stackrel{\phi_{1}}{\dashrightarrow}\left(X_{1},B_{1}\right)\stackrel{\phi_{2}}{\dashrightarrow}\cdots\stackrel{\phi_{n}}{\dashrightarrow}\left(X_{n},B_{n}\right)=\left(X',B'\right),\]
where each $\phi_i$ is either a fibered modification or a reversion.
\item
 If $\phi$ is not an isomorphism then a decomposition as above of minimal length is unique up to isomorphisms of the intermediate pairs occuring in the decomposition. Furthermore, a decomposition of minimal length is characterized by the property that it is \emph{reduced}, which means  that for every $i=1,\dots,n-1$ the induced birational map $\phi_{i+1}\phi_i:(X_{i-1},B_{i-1})\dashrightarrow (X_{i+1},B_{i+1})$ is of minimal length $2$. 
\item A composition $\phi_{i+1}\phi_i:(X_{i-1},B_{i-1})\dashrightarrow (X_{i+1},B_{i+1})$ as above is not reduced if and only if one of the following holds:
\begin{enumerate}
\item
$\phi_i$ and $\phi_{i+1}$ are both fibered modifications;
\item
$\phi_i$ and $\phi_{i+1}$ are both reversions, and $\phi_{i+1}$ and $(\phi_{i})^{-1}$ have the same proper base-point;
\item
$\phi_i$ and $\phi_{i+1}$ are both reversions, $\phi_{i+1}$ and $(\phi_{i})^{-1}$ do not have the same proper base-point  but each irreducible component of $B_{i-1}$ $($equivalently $B_{i+1})$ has self-intersection $\ge -2$;
\end{enumerate}
In case $(a)$, $\phi_{i+1}\phi_i$ is either a fibered modification $($length $1)$ or an automorphism of pairs $($length $0)$. In case $(b)$, $\phi_{i+1}\phi_i$ is an automorphism of pairs $($length $0)$, and in case $(c)$ it is a reversion $($length $1)$.
\end{enumerate}
\end{prop}

Note that starting from a pair $(X,B)$ of type $(0,-1,-a_1,\ldots ,-a_r)$, the type of the pairs which appear in the sequence are either of the same type or of type $(0,-1,-a_r,\ldots ,-a_1)$. In particular, the fact of having an irreducible curve of self-intersection $\le -3$ in the boundary only depends of the surface $X\backslash B$.

\subsection{Graphs of $\mathbb{A}^1$-fibrations and associated graphs of groups}\label{SubSecGraph}
\indent\newline\noindent The existence of the above decomposition of birational maps between standard pairs into sequences of fibered modifications and reversions enables to associate to every normal affine surface $S$ completable by a standard pair an oriented  graph which encodes equivalence classes of $\mathbb{A}^1$-fibrations on $S$ and links between these. This graph $\mathcal{F}_S$ is defined as follows (see \cite[Definition 4.0.5]{first}).

\begin{defn}\label{A1fibGraph} Given a normal affine surface $S$ completable by a standard pair we let $\mathcal{F}_S$ be the oriented graph with the following vertices and edges:

a) A vertex of $\mathcal{F}_S$ is an equivalence class of pairs $(X,B)$ such that $X\setminus B\cong S$, where two $1$-standard pairs $(X_1,B_1,\overline{\pi_1})$, $(X_2,B_2,\overline{\pi_2})$ define the same vertex if and only if the $\mathbb{A}^{1}$-fibered surfaces $(X_1\setminus B_1,\pi_1)$ and $(X_2\setminus B_2,\pi_2)$ are isomorphic.

b) Any arrow of $\mathcal{F}_S$ is an equivalence class of reversions. If $\phi:(X,B)\dasharrow (X',B')$ is a reversion, then the class $[\phi]$ of $\phi$ is an arrow starting from the class $[(X,B)]$ of $(X,B)$ and ending at the class $[(X',B')]$ of $(X',B')$. Two reversions $\phi_1:(X_1,B_1)\dasharrow (X_1',B_1')$ and $\phi_2:(X_2,B_2)\dasharrow (X_2',B_2')$ define the same arrow if and only if there exist isomorphisms  $\theta:(X_1,B_1)\rightarrow (X_2,B_2)$ and  $\theta':(X_1',B_1')\rightarrow (X_2',B_2')$, such that $\phi_2\circ\theta=\theta'\circ\phi_1$.
\end{defn}

\begin{enavant} \label{AutSFS} By definition, a vertex of $\mathcal{F}_S$ represents an equivalence class of $\mathbb{A}^1$-fibrations on $S$, where two $\mathbb{A}^1$-fibrations $\pi:S\rightarrow \mathbb{A}^1$ and $\pi':S\rightarrow \mathbb{A}^1$ are said to be equivalent if there exist automorphisms $\Psi$ and $\psi$ of $S$ and $\mathbb{A}^1$ respectively such that $\pi'\circ \Psi=\psi \circ\pi$. By virtue of \cite[Proposition 4.0.7]{first}, the graph $\mathcal{F}_S$ is connected. Furthermore, if there exists a standard pair $(X,B=F\tr C \tr E)$ completing $S$ for which $B$ has an irreducible component of self-intersection $\leq -3$ then there exists a natural exact sequence 
$$0\rightarrow H\rightarrow \mathrm{Aut}(S) \rightarrow \Pi_1(\mathcal{F}_S)\rightarrow 0$$
 where $H$ is the normal subgroup of the automorphism group $\mathrm{Aut}(S)$ generated by all automorphisms of $\mathbb{A}^1$-fibrations on $S$ and where $\Pi_1(\mathcal{F}_S)$ is the fundamental group of $\mathcal{F}_S$. This implies in particular that $\mathrm{Aut}(S)$ is generated by automorphisms of $\mathbb{A}^1$-fibrations if and only if $\mathcal{F}_S$ is a tree. 

We will show (Proposition~\ref{Prop:AlgGroups}) that for any algebraic subgroup $G$ of $\Aut(S)$, the image in $\Pi_1(\mathcal{F}_S)$ is finite, which implies in particular that if $\Pi_1(\mathcal{F}_S)$ is not countable, then $\Aut(S)$ cannot be generated by a countable set of algebraic subgroups.
\end{enavant}

\begin{enavant} \label{GraphGroups} Under mild hypotheses that are always satisfied for the surfaces considered in the sequel (see \cite[\S 4.0.10]{first} for a discussion), it is possible to equip these graphs $\mathcal{F}_S$ with  additional structures of a graphs of groups (in the sense of \cite[4.4]{Ser}) determined by the choice of : 

$a)$ for any vertex $v$ of $\mathcal{F}_S$, the group $G_v=\Aut(X_v\setminus B_v,\pi_v)$ for a fixed standard pair $(X_v,B_v,\overline{\pi_v})$ in the class $v$. 

$b)$ for any arrow $a$ of $\mathcal{F}_S$,  the group $G_a=\{(\phi,\phi')\in \Aut(X_a,B_a)\times \Aut(X'_a,B'_a)\ |\ r_a\circ \phi=\phi'\circ r_a\}$ for a fixed reversion $r_a:(X_a,B_a,\overline{\pi_a})\stackrel{r_a}{\dasharrow} (X'_a,B'_a,\overline{\pi'_a})$ in the class of $a$, and the injective homomorphism $\rho_a:G_a\rightarrow G_{t(a)}$, $(\phi,\phi'))\mapsto \mu_a\circ \phi'\circ (\mu_a)^{-1}$ for a fixed isomorphism $\mu$ between $(X'_a,B'_a,\overline{\pi'_a})$ and the fixed standard pair on the target vertex $t(a)$ of $a$.  

We further require that the chosen reversion $r_{\bar{a}}$ for the inverse $\bar{a}$ of the arrow $a$ is equal to $(r_a)^{-1}$ and that the structural anti-isomorphism $\bar{}:G_a\rightarrow G_{\bar a}$ is the map $(\phi,\phi')\mapsto (\phi',\phi)$.

\indent\newline
  When such a structure exists on $\mathcal{F}_S$, we established in \cite[Theorem 4.0.11]{first} that the automorphism group $\mathrm{Aut}(S)$ of $S$ is isomorphic to the fundamental group of $\mathcal{F}_S$ as a graph of groups. This means that up to a choice of a base vertex $v$ in $\mathcal{F}_S$, $\mathrm{Aut}(S)$ can be identified with the set of closed paths $g_na_ng_{n-1}\cdots a_2g_2a_1g_1$ where $a_i$ is an arrow from $v_i$ to $v_{i+1}$, $g_i\in G_{v_i}$ and $v_1=v_n=v$  modulo the relations $\rho_a(h)\cdot a=a\cdot \rho_{\bar{a}}(\bar{h})$ and $a\bar{a}=1$ for any arrow $a$ and any $h\in G_a$.

\end{enavant}

\subsection{Actions of algebraic groups on standard pairs}\label{Sec:Action}
Here we derive from the decomposition results some additional informations about algebraic subgroups of automorphism groups of affine surfaces $S$ completable by a standard pair. 

\begin{enavant} Let us first observe that the automorphism group $\Aut(S)$ of an affine surface $S=X\setminus B$ completable by a standard pair $(X,B)$ contains many algebraic subgroups. For instance, it follows from Definition \ref{FiberedModif} (see also \cite[Lemma 5.2.1 or Lemma 2.2.3]{first}) that automorphisms of $S$ preserving the $A^1$-fibration $\pi=\overline{\pi}\mid_S:S=X\setminus B \rightarrow  \mathbb{A}^1$ come as lifts of suitable  triangular automorphisms $\psi$ of $\mathbb{A}^2$ of the form $(x,y)\mapsto (ax+R(y),cy)$ where $a,c\in k^*$ and $R(y)\in k[y]$. It follows in particular the subgroup $\Aut(X\setminus B,\pi)$ of $\Aut(X\setminus B)$ consisting of automorphisms preserving an $\mathbb{A}^1$-fibration $\pi$ is a (countable) increasing union of algebraic subgroups. Furthermore the group $\Aut(X,B)$ of automorphisms of the pair $(X,B)$ is itself algebraic as $B$ is the support of an ample divisor and coincides with the subgroup of $\Aut(X\setminus B,\pi)$ consisting of lifts of affine automorphisms $\psi$ as above.   
\end{enavant}

The following Proposition describes more generally the structure of all possible algebraic subgroups of $\Aut(S)$. 

\begin{prop}\label{Prop:AlgGroups}
Let $S$ be an affine normal surface completable by a standard pair, and let $G\subset \Aut(S)$ be an algebraic subgroup. Then there exists a standard pair $(X,B)$ and an isomorphism $\psi:S\stackrel{\sim}{\rightarrow} X\setminus B$ such that for the conjugate $G^{\psi}=\psi G\psi^{-1}\subset \Aut(X\setminus B)$ of $G$ the following alternative holds:
\begin{enumerate}
\item
 $G^{\psi}$ is a subgroup of $\Aut(X\backslash B,\pi)$;
\item
$G^{\psi}$ contains a reversion $(X,B)\dasharrow (X,B)$ and every other element of $G^{\psi}$ is either a reversion from $(X,B)$ to itself or an element of $\Aut(X,B)$. More precisely, one of the following holds: 
\begin{enumerate}
\item
There exists a $\k$-rational point $p\in B$ such that every reversion in $G^{\psi}$ is centered at $p$, and every element in $G_0^{\psi}=G^{\psi}\cap \Aut(X,B)$ fixes the point $p$. We then have an exact sequence $1\rightarrow G_0^{\psi}\to G^{\psi} \to \mathbb{Z}/2\mathbb{Z}\to 0,$ and $G_0^{\psi}$ is an algebraic group of dimension $\le 2$.
\item
 Every irreducible component of $B$ has self-intersection $\ge -2$ and the contraction $(X,B)\to (Y,D)$ of all irreducible components of negative self-intersection in $B$ conjugates $G^{\psi}$ whence $G$ to a subgroup of $\Aut(Y,D)$, where $Y$ is a projective rational surface and $D\cong \p^1$ is the support of an ample divisor.
\end{enumerate}\end{enumerate}
\end{prop}
\begin{proof}
Up to fixing a standard pair such that $X\backslash B$ is isomorphic to $S$, we may assume from the very beginning that $S=X\setminus B$. Let $G\subset \Aut(S)$ be an algebraic group. According to~\cite{bib:Sum}, there exists a projective surface $Y$ which contains $S$ as an open dense subset, and such that $G$ extends to a group of automorphisms of $Y$. Since the induced birational map $X\dasharrow Y$  has only a finite number $r$ of base-points (including infinitely near ones), it follows that every element element $g\in G$ considered as a birational self-map of $X$ has at most $2r$ base-points (again including infinitely near ones).  
This implies in turn that the number $n$ of fibered modifications and reversions occuring in a minimal decomposition $g=\phi_n \dots \phi_1$ (see Proposition~\ref{mainfirst}) of every element $g\in G$ considered as a birational self-map of $(X,B)$ is bounded. We call this integer $n$  the \emph{length} of $g$ (by convention an automorphism of the pair $(X,B)$ has length $0$). We denote $m=m(G)$ the maximum length of elements of $G$.

a) If $m\le 1$ and no element of $G$ is a reversion, then $G\subset \Aut(X\setminus B,\pi)$ and we get $(1)$.

b) Otherwise, if $m\le 1$ and $G$ contains a reversion, then $m=1$ and all elements of $G$ are either reversions or automorphisms of $(X,B)$ because the composition of a fibered modification and a reversion has always length $2$. If there exists a $\k$-rational point $p\in B$ such that every reversion in $G$ is centered at $p$ and every element in $G_0=G\cap \Aut(X,B)$ fixes the point $p$, then the product of any two reversions in $G$ is an automorphism (Proposition~\ref{mainfirst} 3)c) ). It follows that $G/G_0\cong \mathbb{Z}/2\mathbb{Z}$, which gives $2a)$. Otherwise, if $G$ contains two reversions $\phi,\phi'$ with distinct proper base-points then since $\phi'\phi^{-1}$ has length at most $1$ by hypothesis, it cannot be reduced. By Proposition~\ref{mainfirst} 3)c), it follows that all components of $B$ have self-intersection $\ge -2$. But in this case, reversions simply correspond to contracting all components of $B=F\tr C\tr E$ negative self-intersections to a point $p$ on the proper transform $D$ of $F$ and then blowing-up a new chain of the same type starting from a point $p'\in D$ distinct from $p$. The conjugation by the corresponding contraction $(X,B)\to (Y,D)$ identifies $G$ with a subgroup of $\Aut(Y,D)$, which gives $2b)$.

To complete the proof, it remains to show that we can always reduce by an appropriate conjugation to the case $m\le 1$. If $m\ge 2$, then we consider an element $g\in G$ of length $m$ and we fix a reduced decomposition $g=\phi_m\cdots \phi_1$. Writing $\phi_1\colon (X,B)\dasharrow (X_1,B_1)$ and arguing by induction on $m$, it is enough to show that the  length of every element in $\phi_1G \phi_1^{-1}$ considered as a group of birational self-maps $(X_1,B_1)\dasharrow (X_1,B_1)$ is at most $m-1$. Given an element $h\in G$ of lenght $n\geq 0$, we have the following possibilities:

1) If $n=0$ then $h$ is an automorphism of $(X,B)$. If it fixes the proper base-point of $\phi_1$, then $\phi_1 h \phi_1^{-1}$ is not reduced whence has length at most $1\le m-1$. Otherwise, if the proper base-point of $\phi_1$ is not a fixed point of $h$, then since by hypothesis $g h g^{-1}=\phi_m \cdots \phi_1 h \phi_1^{-1}\cdots \phi_m^{-1}$ has length $\le m$, it cannot be reduced. It follows necessarily that $\phi_1 h \phi_1^{-1}$ is not reduced which implies in turn that $\phi_1$ is a reversion and that every irreducible curve in $B$ has self-intersection $\ge -2$ (using again case $3)c)$ of Proposition~\ref{mainfirst}). In this case, $\phi_1 h \phi_1^{-1}$ is a non-reduced composition of two reversions, and has thus length at most  $1\le m-1$.

2) If $h$ has length $n\ge 1$, then we consider a reduced decomposition $h=\psi_n\cdots \psi_1$ of $h$ into fibered modifications and reversions. Since $gh^{-1}=\phi_m\cdots \phi_1\psi_1^{-1}\cdots \psi_n^{-1}$ is not reduced, then so is $\nu=\phi_1\psi_1^{-1}$, which has thus length $\le 1$. This implies that $\phi_1$ and $\psi_1$ are both fibered modifications or both reversions. 

$\quad$ a) If $n=1$ and $\phi_1$ and $h=\psi_1$ are both fibered modifications, then $\phi_1 h \phi_1^{-1}$ is a fibered modification or an automorphism of the pair $(X_1,B_1)$, and hence has length $\le 1\le m-1$. Otherwise, if $n=1$ and $\phi_1$ and $h=\psi_1$ are both reversions, then either $\nu$ is an isomorphism of pairs, in which case  $\phi_1 h \phi_1^{-1}$ has length $1$, or it is a reversion and then all irreducibles curves of $B$ and $B'$ have self-intersection $\geq -2$. This implies that $\phi_1 h \phi_1^{-1}$ is again a reversion or an automorphism, hence has length $\le 1$.

$\quad$ b) Finally, if $n\ge 2$, then the composition $\phi_2\nu \psi_2^{-1}$ is not reduced which implies that $\nu$ is an isomorphism  of pairs. Replacing $h$ with $h^{-1}$, we also conclude that $\phi_1 \psi_n$ is an isomorphism of pairs. This implies that $\phi_1 h \phi_1^{-1}$ has length at most $m-2$ and completes the proof.
\end{proof}

\begin{rem}\label{RemaLoop1}
Writing $S=X\setminus B$, Proposition~\ref{Prop:AlgGroups} implies that the image of an algebraic subgroup of $\Aut(S)$ under the morphism $\mathrm{Aut}(S) \rightarrow \Pi_1(\mathcal{F}_S)$ described in \ref{AutSFS} is very special: it consists of a path of the form $\varphi^{-1} \sigma \varphi$, where $\varphi$ is a path from $[(X,B)]$ to another vertex $[(X',B')]$ and $\sigma$ is either trivial or a loop of length $1$ based at the vertex $[(X',B')]$ representing a reversion $(X',B')\dashrightarrow (X'',B'')$ between two isomorphic pairs (see Definition \ref{AutSFS}). This implies in turn that in most cases, in particular whenever the type $(0,-1,-a_1,\dots,-a_r)$ of $B$ does not satisfy  $(a_1,\dots,a_r)=(a_r,\dots,a_1)$, the image in $\Pi_1(\mathcal{F}_S)$ of an algebraic subgroup of $\Aut(S)$ is trivial.
\end{rem}

\section{Affine surfaces completable by a standard pair of type $(0,-1,-a,-b)$}

In this section we classify all models of standard pairs $(X,B)$ of type $(0,-1,-a,-b)$, $a,b\geq 2$ for which $X\setminus B$ is a normal affine surface. 

\subsection{Construction of standard pairs}
\indent\newline\noindent Here we construct standard pairs of type $(0,-1,-a,-b)$ in terms of the base points of the birational morphism $\eta\colon Y\to \mathbb{F}_1$ from their minimal resolution of singularities as in \ref{AffineFiberedDef}.

\begin{enavant}\label{SetUpF1} In what follows, we consider $\mathbb{F}_1$ embedded into $\Pn\times\mathbb{P}^1$ as \[\mathbb{F}_1=\{\left((x:y:z),(s:t)\right) \subset \Pn\times\mathbb{P}^1\ |\ yt=zs\};\] the projection on the first factor yields the birational morphism $\tau:\mathbb{F}_1\rightarrow \Pn$ which is the blow-up of $(1:0:0)\in\Pn$ and the projection on the second factor yields the $\mathbb{P}^1$-bundle  $\rho: \mathbb{F}_1\rightarrow\mathbb{P}^1$, corresponding to the projection of $\p^2$ from $(1:0:0)$. We denote by $F,L\subset \Pn$ the lines with equations $z=0$ and $y=0$ respectively. We also call $F,L\subset \mathbb{F}_1$ their proper transforms on $\mathbb{F}_1$, and denote by $C\subset \mathbb{F}_1$ the exceptional curve $\tau^{-1}((1:0:0))=(1:0:0)\times\mathbb{P}^1$. The affine line $L\setminus C\subset \mathbb{F}_1$ and its image $L\setminus (1:0:0)\subset \Pn$ will be called $L_0$. The morphism $\mathbb{A}^2={\rm Spec}(k[x,y])\rightarrow  \Pn\times\mathbb{P}^1$, $(x,y)\to ((x:y:1),(y:1))$ induces an open embedding of $\mathbb{A}^2$ into  $\mathbb{F}_1$ as the complement of $F\cup C$ for which $L_0$ coincides with the line $y=0$.  
With this notation each of the points blown-up by $\eta$ belongs -- as proper or infinitely near point -- to the affine line $L_0$, is defined over $\kk$ but not necessarily over $\k$; however, the set of all points blown-up by $\eta$ is defined over $\k$.
\end{enavant}

\begin{enavant}\label{PairConstruction} Given polynomials $P,Q\in\k[w]$ of degrees $a-1$ and $b-1$ respectively, we define a birational morphism $\eta_{P,Q}:Y\rightarrow \mathbb{F}_1$ which blows-up $a+b-1$ points that belong, as proper or infinitely near points, to $L\setminus C$. The morphism $\eta_{P,Q}$ is equal to $\eta_{wP}\circ \epsilon_{P,Q}$, where $\eta_{wP}$ and $\epsilon_{P,Q}$ are birational morphisms which blow-up respectively $a$ and $b-1$ points, defined as follows:
\begin{enumerate} 

\item The map $\eta_{wP}:W\rightarrow \mathbb{F}_1$ is a birational morphism associated to $wP(w)$ which blows-up $a$ points as follows. Let $\alpha_0=0,\alpha_1,\ldots, \alpha_l$ be the distinct roots of $wP(w)$ with respective multiplicities $r_0+1>0$ and $r_i>0$, $i=1,\ldots, l$. Then  $\eta_{wP}$ is obtained by first blowing each of the points $(\alpha_i,0)\in (L\setminus C)(\kk)$ and then $r_i-1$ other points in their respective infinitesimal neighborhoods, each one belonging to the proper transform of $L$. We denote by $A_i$ the last exceptional divisor produced by the this sequence of blow-ups over each point $(\alpha_i,0)$, $i=0,\ldots, l$ and by $\mathcal{A}_i$ the union of the $r_i-1$ other components in the inverse image of $(\alpha_i,0)$ of self-intersection $(-2)$. The $r_0$ blow-ups over above $(0,0)$ are described locally by $(u,v)\mapsto (x,y)=(u,u^{r_0}v)$; the curve $\{v=0\}$ corresponds to the proper transform of $L$. The last step consists of the blow-up of $(u,v)=(0,0)$ with exceptional divisor $E_2$.  

\item The map $\epsilon_{P,Q}:Y\rightarrow W$ is the blow-up of $b-1$ points as follows. If $P(0)\neq 0$ then we let $\beta_1,\ldots, \beta_m$ be the distinct roots of $Q$ in $\kk$. Otherwise if $P(0)=0$ then we let $\beta_0=0$ and we denote by $\beta_1,\ldots, \beta_m$ the non-zero distinct roots of $Q$ in $\kk$. In each case, we denote by $s_j$ the multiplicity of $\beta_j$ as a root of $Q$. Then $\epsilon_{P,Q}$ consists for every $j=0,1,\ldots, m$ of the blow-up of the point in $E_2\setminus L(\kk)$ corresponding to the direction $u+\beta_j v=0$ is blown-up followed by the blow-up of $s_j-1$ other points in its infinitesimal neighborhood, each one belonging to the proper transform of $E_2$. For every $j=0,\ldots, m$, we denote respectively by $B_j$ and $\mathcal{B}_j$ the last exceptional divisor and the union of the $s_j-1$ other components of self-intersection $-2$ of the exceptional locus of $\epsilon_{P,Q}$ over the corresponding point of $E_2$.
\end{enumerate}

We denote by $E_1$ the proper transform of $L$ on the smooth projective surface $Y$ obtained by the above procedure and we let  $E=E_1\tr E_2$. Then the contraction of every exceptional divisor of $\eta_{P,Q}$ not intersecting $E$ yields a birational morphism $\mu_{P,Q}:Y\rightarrow X$ to a normal projective surface $X$ for which the zigzag $B=F\tr C\tr E_1\tr E_2$ has type $(0,-1,-a,-b)$. The definition of $\eta_{wP}$ and $\epsilon_{P,Q}$ implies that when $P(0)=0$ the direction of the line $u=0$ is a special point in $E_2\subset W$ corresponding to the curve contracted by $\eta_{wP}$ which intersects $E_2$. Furthermore, this point is blown-up by $\epsilon_{P,Q}$ if and only if $Q(0)=0$.

 This leads to the following three possible cases below: 
\end{enavant}

\begin{enavant}{\bf Case  $\mathbf{I}$}: $P(0)\not=0$. The unique degenerate fiber $F_0$ of the rational pencil $\bar{\pi}:X\rightarrow \mathbb{P}^1$ defined by the proper transform of $F$ consists of the total transform $(\mu_{P,Q})_*(\eta_{P,Q})^*L$ of $L$. The multiplicities of the roots of $P$ and $Q$ in $\kk$ coincide with that of the  corresponding irreducible components of $F_0$. Furthermore, each multiple root of $P$ (resp. of $Q$) yields a cyclic quotient rational double point of $X$ of order $r_i$ (resp. $s_j$) supported on the corresponding irreducible component of $F_0$.  
\begin{figure}[ht]
\begin{pspicture}(-2.6,0.5)(3,3.5)

\begin{pspicture}(0,0)(3,1.5)
{\darkgray
\psline[linecolor=darkgray](-1,1)(-0.75,1.7)\rput(-0.75,1.7){\textbullet}\rput(-0.6,1.55){\scriptsize $-1$}\rput(-0.65,1.9){\small $B_m$}
\psline[linecolor=darkgray,linearc=2](-1,1)(-1.05,1.9)(-0.8,2.6)\rput(-0.8,2.6){\textbullet}\rput(-0.65,2.45){\scriptsize $-1$}\rput(-0.7,2.8){\small $B_1$}
\psline[linecolor=darkgray](-0.8,2.6)(-0.4,2.6)
\psframe[linecolor=darkgray](-0.4,2.4)(0.4,2.8)
\psline[linecolor=darkgray](-0.8,1.7)(-0.4,1.7)
\psframe[linecolor=darkgray](-0.4,1.5)(0.4,1.9)
\rput(0,1.7){\scriptsize $s_m\!-\!1$}\rput(0.2,2.07){\small $\mathcal{B}_m$}
\rput(-0.1,2.3){$\vdots$}
\rput(0,2.6){\scriptsize $s_1\!-\!1$}\rput(0.2,2.97){\small $\mathcal{B}_1$}
\psline[linecolor=darkgray](1,1)(1.25,1.7)\rput(1.25,1.7){\textbullet}\rput(1.4,1.55){\scriptsize $-1$}\rput(1.35,1.9){\small $A_l$}
\psline[linecolor=darkgray,linearc=2](1,1)(0.95,1.9)(1.2,2.6)\rput(1.2,2.6){\textbullet}\rput(1.35,2.45){\scriptsize $-1$}\rput(1.3,2.8){\small $A_1$}
\psline[linecolor=darkgray](1.2,2.6)(1.6,2.6)
\psframe[linecolor=darkgray](1.6,2.4)(2.4,2.8)
\psline[linecolor=darkgray](1.2,1.7)(1.6,1.7)
\psframe[linecolor=darkgray](1.6,1.5)(2.4,1.9)
\rput(2,1.7){\scriptsize $r_l\!-\!1$}\rput(2.2,2.07){\small $\mathcal{A}_l$}
\rput(1.9,2.3){$\vdots$}
\rput(2,2.6){\scriptsize $r_1\!-\!1$}\rput(2.2,2.97){\small $\mathcal{A}_1$}
}
\psline(-1,1)(3,1)
\rput(0.82,1.2){{\small $E_1$}}
\rput(-1.25,1.2){{\small $E_2$}}
\rput(2,1.25){{\small $C$}}
\rput(3.18,1.15){{\small $F$}}
\rput(-1,1){\textbullet}\rput(-1.05,0.75){{\scriptsize $-b$}}
\rput(1,1){\textbullet}\rput(0.95,0.75){{\scriptsize $-a$}}
\rput(2,1){\textbullet}\rput(1.95,0.75){{\scriptsize $-1$}}
\rput(3,1){\textbullet}\rput(3,0.75){{\scriptsize $0$}}
\end{pspicture}

\psline{->}(0.6,1)(1.4,1)
\rput(1,1.3){$\epsilon_{P,Q}$}

\begin{pspicture}(-2,0)(3,1.5)
{\darkgray
\psline[linecolor=darkgray](1,1)(1.25,1.7)\rput(1.25,1.7){\textbullet}\rput(1.4,1.55){\scriptsize $-1$}\rput(1.35,1.9){\small $A_l$}
\psline[linecolor=darkgray,linearc=2](1,1)(0.95,1.9)(1.2,2.6)\rput(1.2,2.6){\textbullet}\rput(1.35,2.45){\scriptsize $-1$}\rput(1.3,2.8){\small $A_1$}
\psline[linecolor=darkgray](1.2,2.6)(1.6,2.6)
\psframe[linecolor=darkgray](1.6,2.4)(2.4,2.8)
\psline[linecolor=darkgray](1.2,1.7)(1.6,1.7)
\psframe[linecolor=darkgray](1.6,1.5)(2.4,1.9)
\rput(2,1.7){\scriptsize $r_l\!-\!1$}\rput(2.2,2.07){\small $\mathcal{A}_l$}
\rput(1.9,2.3){$\vdots$}
\rput(2,2.6){\scriptsize $r_1\!-\!1$}\rput(2.2,2.97){\small $\mathcal{A}_1$}
}
\psline(0,1)(3,1)
\rput(0.82,1.2){{\small $E_1$}}
\rput(0,1.25){{\small $E_2$}}
\rput(2,1.25){{\small $C$}}
\rput(3.18,1.15){{\small $F$}}
\rput(0,1){\textbullet}\rput(-0.05,0.75){{\scriptsize $-1$}}
\rput(1,1){\textbullet}\rput(0.95,0.75){{\scriptsize $-a$}}
\rput(2,1){\textbullet}\rput(1.95,0.75){{\scriptsize $-1$}}
\rput(3,1){\textbullet}\rput(3,0.75){{\scriptsize $0$}}
\end{pspicture}

\psline{->}(0.6,1)(1.4,1)
\rput(1,1.2){$\eta_{wP}$}

\begin{pspicture}(-1,0)(3,1.5)
\psline(1,1)(3,1)
\rput(1,1){\textbullet}\rput(1,1.3){{\small $F$}}\rput(0.95,0.8){{\scriptsize 0}}
\rput(2,1){\textbullet}\rput(2,1.3){{\small $C$}}\rput(1.95,0.8){{\scriptsize $-1$}}
\rput(3,1){\textbullet}\rput(3,1.3){{\small $L$}}\rput(2.95,0.8){{\scriptsize $0$}}
\end{pspicture}

\end{pspicture}
\caption{ The morphisms $(Y,B)\stackrel{\epsilon_{P,Q}}{\rightarrow}(W,B))\stackrel{\eta_{wP}}{\rightarrow}(\mathbb{F}_1,F\tr C\tr L)$ when  $P(0)\not=0$. 
A block with label $t$ consists of a zigzag of $t$ $(-2)$-curves.\label{Fig:DoubleBlowUpr0eq0}}
\end{figure}
\end{enavant}

\begin{enavant} {\bf Case  $\mathbf{II}$: } $P(0)=0$ and $Q(0)\not=0$. In the unique degenerate fiber $F_0=(\mu_{P,Q})_*(\eta_{P,Q})^*L$ of the induced rational pencil $\bar{\pi}:X\rightarrow \mathbb{P}^1$, the multiplicities of the  roots of $P$ in $\kk$ coincide with that of the corresponding irreducible components $A_i$, $i=0,\ldots,l$ whereas each irreducible component $B_j$, $j=1,\ldots, m$ corresponding to a root $\beta_j$ of $Q$ has multiplicity $(r_0+1)s_j$ in $F_0$.  Similarly as in case I, each multiple root of $P$ (resp. of $Q$) yields a cyclic quotient rational double point of $X$ supported on the corresponding irreducible component of $F_0$.  

\begin{figure}[ht]
\begin{pspicture}(-2.6,0.5)(3,4)

\begin{pspicture}(0,0)(3,1.5)
{\darkgray
\psline[linecolor=darkgray](-1,1)(-0.75,1.7)\rput(-0.75,1.7){\textbullet}\rput(-0.6,1.55){\scriptsize $-1$}\rput(-0.65,1.9){\small $B_m$}
\psline[linecolor=darkgray,linearc=2](-1,1)(-1.05,1.9)(-0.8,2.6)\rput(-0.8,2.6){\textbullet}\rput(-0.65,2.45){\scriptsize $-1$}\rput(-0.7,2.8){\small $B_1$}\psline[linecolor=darkgray,linearc=2](-1,1)(-1.15,2.2)(-0.8,3.4)
\rput(-0.8,3.4){\textbullet}\rput(-0.65,3.25){\scriptsize $-2$}\rput(-0.8,3.6){\small $A_0$}
\psline[linecolor=darkgray](-0.8,2.6)(-0.4,2.6)
\psframe[linecolor=darkgray](-0.4,2.4)(0.4,2.8)
\psline[linecolor=darkgray](-0.8,3.4)(-0.4,3.4)
\psframe[linecolor=darkgray](-0.4,3.2)(0.4,3.6)
\psline[linecolor=darkgray](-0.8,1.7)(-0.4,1.7)
\psframe[linecolor=darkgray](-0.4,1.5)(0.4,1.9)
\rput(0,1.7){\scriptsize $s_m\!-\!1$}\rput(0.2,2.07){\small $\mathcal{B}_m$}
\rput(-0.1,2.3){$\vdots$}
\rput(0,2.6){\scriptsize $s_1\!-\!1$}\rput(0.2,2.97){\small $\mathcal{B}_1$}
\rput(0,3.4){\scriptsize $r_0\!-\!1$}\rput(0.2,3.77){\small $\mathcal{A}_0$}
\psline[linecolor=darkgray](1,1)(1.25,1.7)\rput(1.25,1.7){\textbullet}\rput(1.4,1.55){\scriptsize $-1$}\rput(1.35,1.9){\small $A_l$}
\psline[linecolor=darkgray,linearc=2](1,1)(0.95,1.9)(1.2,2.6)\rput(1.2,2.6){\textbullet}\rput(1.35,2.45){\scriptsize $-1$}\rput(1.3,2.8){\small $A_1$}
\psline[linecolor=darkgray](1.2,2.6)(1.6,2.6)
\psframe[linecolor=darkgray](1.6,2.4)(2.4,2.8)
\psline[linecolor=darkgray](1.2,1.7)(1.6,1.7)
\psframe[linecolor=darkgray](1.6,1.5)(2.4,1.9)
\rput(2,1.7){\scriptsize $r_l\!-\!1$}\rput(2.2,2.07){\small $\mathcal{A}_l$}
\rput(1.9,2.3){$\vdots$}
\rput(2,2.6){\scriptsize $r_1\!-\!1$}\rput(2.2,2.97){\small $\mathcal{A}_1$}
}

\psline(-1,1)(3,1)
\rput(0.82,1.2){{\small $E_1$}}
\rput(-1.18,1.2){{\small $E_2$}}
\rput(2,1.25){{\small $C$}}
\rput(3.18,1.15){{\small $F$}}
\rput(-1,1){\textbullet}\rput(-1.05,0.75){{\scriptsize $-b$}}
\rput(1,1){\textbullet}\rput(0.95,0.75){{\scriptsize $-a$}}
\rput(2,1){\textbullet}\rput(1.95,0.75){{\scriptsize $-1$}}
\rput(3,1){\textbullet}\rput(3,0.75){{\scriptsize $0$}}
\end{pspicture}

\psline{->}(0.6,1)(1.4,1)
\rput(1,1.3){$\epsilon_{P,Q}$}

\begin{pspicture}(-3,0)(3,1.5)
{\darkgray
\psline[linecolor=darkgray,linearc=2](-1,1)(-1.15,2.2)(-0.8,3.4)
\rput(-0.8,3.4){\textbullet}\rput(-0.65,3.25){\scriptsize $-2$}\rput(-0.8,3.6){\small $A_0$}
\psline[linecolor=darkgray](-0.8,3.4)(-0.4,3.4)
\psframe[linecolor=darkgray](-0.4,3.2)(0.4,3.6)
\rput(0,3.4){\scriptsize $r_0\!-\!1$}\rput(0.2,3.77){\small $\mathcal{A}_0$}
\psline[linecolor=darkgray](1,1)(1.25,1.7)\rput(1.25,1.7){\textbullet}\rput(1.4,1.55){\scriptsize $-1$}\rput(1.35,1.9){\small $A_l$}
\psline[linecolor=darkgray,linearc=2](1,1)(0.95,1.9)(1.2,2.6)\rput(1.2,2.6){\textbullet}\rput(1.35,2.45){\scriptsize $-1$}\rput(1.3,2.8){\small $A_1$}
\psline[linecolor=darkgray](1.2,2.6)(1.6,2.6)
\psframe[linecolor=darkgray](1.6,2.4)(2.4,2.8)
\psline[linecolor=darkgray](1.2,1.7)(1.6,1.7)
\psframe[linecolor=darkgray](1.6,1.5)(2.4,1.9)
\rput(2,1.7){\scriptsize $r_l\!-\!1$}\rput(2.2,2.07){\small $\mathcal{A}_l$}
\rput(1.9,2.3){$\vdots$}
\rput(2,2.6){\scriptsize $r_1\!-\!1$}\rput(2.2,2.97){\small $\mathcal{A}_1$}
}

\psline(-1,1)(3,1)
\rput(0.82,1.2){{\small $E_1$}}
\rput(-1.18,1.2){{\small $E_2$}}
\rput(2,1.25){{\small $C$}}
\rput(3.18,1.15){{\small $F$}}
\rput(-1,1){\textbullet}\rput(-1.05,0.75){{\scriptsize $-1$}}
\rput(1,1){\textbullet}\rput(0.95,0.75){{\scriptsize $-a$}}
\rput(2,1){\textbullet}\rput(1.95,0.75){{\scriptsize $-1$}}
\rput(3,1){\textbullet}\rput(3,0.75){{\scriptsize $0$}}
\end{pspicture}

\psline{->}(0.6,1)(1.4,1)
\rput(1,1.2){$\eta_{wP}$}

\begin{pspicture}(-1,0)(3,1.5)
\psline(1,1)(3,1)
\rput(1,1){\textbullet}\rput(1,1.3){{\small $F$}}\rput(0.95,0.8){{\scriptsize 0}}
\rput(2,1){\textbullet}\rput(2,1.3){{\small $C$}}\rput(1.95,0.8){{\scriptsize $-1$}}
\rput(3,1){\textbullet}\rput(3,1.3){{\small $L$}}\rput(2.95,0.8){{\scriptsize $0$}}
\end{pspicture}

\end{pspicture}

\caption{ The morphisms $(Y,B)\stackrel{\epsilon_{P,Q}}{\rightarrow}(W,B))\stackrel{\eta_{wP}}{\rightarrow}(\mathbb{F}_1,F\tr C\tr L)$ when  $P(0)=0$ and $Q(0)\not=0$. A block with label $t$ consists of a zigzag of $t$ $(-2)$-curves.\label{Fig:DoubleBlowUpr0geq1}}

\end{figure}
\end{enavant}

\begin{enavant} {\bf Case  $\mathbf{III}$:} $P(0)=Q(0)=0$. In this model again, the multiplicity of the non-zero roots of $P$ coincide with that of the corresponding irreducible components of the degenerate fibre $F_0=(\mu_{P,Q})_*(\eta_{P,Q})^*L$ of $\bar{\pi}:X\rightarrow \mathbb{P}^1$, each supporting a cyclic quotient rational double point of order $r_i$. A component $B_j$ of $F_0$ corresponding to a non-zero root of $Q$ has multiplicity $(r_0+1)s_j$ and support a cyclic quotient rational double point of order $s_j$. Finally the irreducible component $B_0$ of $F_0$ corresponding to the common root $0$ of $P$ and $Q$ has multiplicity $(s_0+1)(r_0+1)-1$. Furthermore, it supports a singular point of $X$ whose minimal resolution is a zigzag $\mathcal{B}_0\tr E_3 \tr \mathcal{A}_0$ where $E_3$ is rational curve with self-intersection $-3$ and where $\mathcal{B}_0$ and $\mathcal{A}_0$ are chains of $s_0-1$ and $r_0-1$ $(-2)$-curves respectively. 

\begin{figure}[ht]
\begin{pspicture}(-2.6,0.5)(3,4)

\begin{pspicture}(0,00)(3,1.5)
{\darkgray
\psline[linecolor=darkgray](-1,1)(-0.75,1.7)\rput(-0.75,1.7){\textbullet}\rput(-0.6,1.55){\scriptsize $-1$}\rput(-0.65,1.9){\small $B_m$}
\psline[linecolor=darkgray,linearc=2](-1,1)(-1.05,1.9)(-0.8,2.6)\rput(-0.8,2.6){\textbullet}\rput(-0.65,2.45){\scriptsize $-1$}\rput(-0.7,2.8){\small $B_1$}\psline[linecolor=darkgray,linearc=2](-1,1)(-1.15,2.2)(-0.8,3.4)
\rput(-0.8,3.4){\textbullet}\rput(-0.65,3.25){\scriptsize $-1$}\rput(-0.8,3.6){\small $B_0$}
\rput(0.8,3.4){\textbullet}\rput(0.8,3.25){\scriptsize $-3$}\rput(0.8,3.6){\small $A_0$}
\psline[linecolor=darkgray](-0.8,2.6)(-0.4,2.6)
\psframe[linecolor=darkgray](-0.4,2.4)(0.4,2.8)
\psline[linecolor=darkgray](-0.8,3.4)(-0.4,3.4)
\psframe[linecolor=darkgray](-0.4,3.2)(0.4,3.6)
\psline[linecolor=darkgray](0.4,3.4)(1.2,3.4)
\psframe[linecolor=darkgray](1.2,3.2)(2,3.6)
\psline[linecolor=darkgray](-0.8,1.7)(-0.4,1.7)
\psframe[linecolor=darkgray](-0.4,1.5)(0.4,1.9)
\rput(0,1.7){\scriptsize $s_m\!-\!1$}\rput(0.2,2.07){\small $\mathcal{B}_m$}
\rput(-0.1,2.3){$\vdots$}
\rput(0,2.6){\scriptsize $s_1\!-\!1$}\rput(0.2,2.97){\small $\mathcal{B}_1$}
\rput(0,3.4){\scriptsize $s_0\!-\!1$}\rput(0.2,3.77){\small $\mathcal{B}_0$}
\rput(1.6,3.4){\scriptsize $r_0\!-\!1$}\rput(1.8,3.77){\small $\mathcal{A}_0$}
\psline[linecolor=darkgray](1,1)(1.25,1.7)\rput(1.25,1.7){\textbullet}\rput(1.4,1.55){\scriptsize $-1$}\rput(1.35,1.9){\small $A_l$}
\psline[linecolor=darkgray,linearc=2](1,1)(0.95,1.9)(1.2,2.6)\rput(1.2,2.6){\textbullet}\rput(1.35,2.45){\scriptsize $-1$}\rput(1.3,2.8){\small $A_1$}

\psline[linecolor=darkgray](1.2,2.6)(1.6,2.6)
\psframe[linecolor=darkgray](1.6,2.4)(2.4,2.8)
\psline[linecolor=darkgray](1.2,1.7)(1.6,1.7)
\psframe[linecolor=darkgray](1.6,1.5)(2.4,1.9)
\rput(2,1.7){\scriptsize $r_l\!-\!1$}\rput(2.2,2.07){\small $\mathcal{A}_l$}
\rput(1.9,2.3){$\vdots$}
\rput(2,2.6){\scriptsize $r_1\!-\!1$}\rput(2.2,2.97){\small $\mathcal{A}_1$}
}
\psline(-1,1)(3,1)
\rput(0.82,1.2){{\small $E_1$}}
\rput(-1.18,1.2){{\small $E_2$}}
\rput(2,1.25){{\small $C$}}
\rput(3.18,1.15){{\small $F$}}
\rput(-1,1){\textbullet}\rput(-1.05,0.75){{\scriptsize $-b$}}
\rput(1,1){\textbullet}\rput(0.95,0.75){{\scriptsize $-a$}}
\rput(2,1){\textbullet}\rput(1.95,0.75){{\scriptsize $-1$}}
\rput(3,1){\textbullet}\rput(3,0.75){{\scriptsize $0$}}
\end{pspicture}

\psline{->}(0.6,1)(1.4,1)
\rput(1,1.3){$\epsilon_{P,Q}$}

\begin{pspicture}(-3,0)(3,1.5)
{\darkgray
\psline[linecolor=darkgray,linearc=2](-1,1)(-1.15,2.2)(-0.8,3.4)
\rput(-0.8,3.4){\textbullet}\rput(-0.65,3.25){\scriptsize $-2$}\rput(-0.8,3.6){\small $A_0$}
\psline[linecolor=darkgray](-0.8,3.4)(-0.4,3.4)
\psframe[linecolor=darkgray](-0.4,3.2)(0.4,3.6)
\rput(0,3.4){\scriptsize $r_0\!-\!1$}\rput(0.2,3.77){\small $\mathcal{A}_0$}
\psline[linecolor=darkgray](1,1)(1.25,1.7)\rput(1.25,1.7){\textbullet}\rput(1.4,1.55){\scriptsize $-1$}\rput(1.35,1.9){\small $A_l$}
\psline[linecolor=darkgray,linearc=2](1,1)(0.95,1.9)(1.2,2.6)\rput(1.2,2.6){\textbullet}\rput(1.35,2.45){\scriptsize $-1$}\rput(1.3,2.8){\small $A_1$}
\psline[linecolor=darkgray](1.2,2.6)(1.6,2.6)
\psframe[linecolor=darkgray](1.6,2.4)(2.4,2.8)
\psline[linecolor=darkgray](1.2,1.7)(1.6,1.7)
\psframe[linecolor=darkgray](1.6,1.5)(2.4,1.9)
\rput(2,1.7){\scriptsize $r_l\!-\!1$}\rput(2.2,2.07){\small $\mathcal{A}_l$}
\rput(1.9,2.3){$\vdots$}
\rput(2,2.6){\scriptsize $r_1\!-\!1$}\rput(2.2,2.97){\small $\mathcal{A}_1$}
}

\psline(-1,1)(3,1)
\rput(0.82,1.2){{\small $E_1$}}
\rput(-1.18,1.2){{\small $E_2$}}
\rput(2,1.25){{\small $C$}}
\rput(3.18,1.15){{\small $F$}}
\rput(-1,1){\textbullet}\rput(-1.05,0.75){{\scriptsize $-1$}}
\rput(1,1){\textbullet}\rput(0.95,0.75){{\scriptsize $-a$}}
\rput(2,1){\textbullet}\rput(1.95,0.75){{\scriptsize $-1$}}
\rput(3,1){\textbullet}\rput(3,0.75){{\scriptsize $0$}}
\end{pspicture}

\psline{->}(0.6,1)(1.4,1)
\rput(1,1.2){$\eta_{wP}$}

\begin{pspicture}(-1,0)(3,1.5)
\psline(1,1)(3,1)
\rput(1,1){\textbullet}\rput(1,1.3){{\small $F$}}\rput(0.95,0.8){{\scriptsize 0}}
\rput(2,1){\textbullet}\rput(2,1.3){{\small $C$}}\rput(1.95,0.8){{\scriptsize $-1$}}
\rput(3,1){\textbullet}\rput(3,1.3){{\small $L$}}\rput(2.95,0.8){{\scriptsize $0$}}
\end{pspicture}

\end{pspicture}

\caption{ The morphisms $(Y,B)\stackrel{\epsilon_{P,Q}}{\rightarrow}(W,B))\stackrel{\eta_{wP}}{\rightarrow}(\mathbb{F}_1,F\tr C\tr L)$ when  $P(0)=Q(0)=0$. A block with label $t$ consists of a zigzag of $t$ $(-2)$-curves.\label{Fig:DoubleBlowUpr0geq2}}
\end{figure}

\end{enavant}

\begin{rem} Case III always leads to a singular surface $X\setminus B$ while in case I and II, the resulting affine surface $X\setminus B$ is smooth if and only if the polynomials $wP(w)$ and $Q(w)$ both have simple roots in $\kk$. The induced $\mathbb{A}^1$-fibration $\pi=\bar{\pi}\mid_{X\setminus B}:X\setminus B\rightarrow \mathbb{A}^1$ has unique degenerate fibre $\pi^{-1}(0)$. If $X\setminus B$ is smooth then, the latter is reduced in case I whereas in case II each root of $Q$ gives rise to an irreducible component of $\pi^{-1}(0)$ of multiplicity two. 
\end{rem}

\begin{enavant}\label{SurfEquations} In each of the above three cases, it follows from the construction that the quasi-projective surface $S=X\setminus B$ does not contain any complete curve. Furthermore, one checks for instance that the divisor $D=4abF+3abC+2bE_1+E_2$ has positive intersection with its irreducible components and positive self-intersection. Hence $B$ is the support of an ample divisor by virtue of the Nakai-Moishezon criterion and so $S$ is a normal affine surface. 

The contraction in the intermediate projective surface $W$ of every exceptional divisor of $\eta_{wP}$ not intersecting $E_1$ yields a birational morphism $\mu_{wP}:W\rightarrow X'$ to a normal projective surface $X'$ for which the zigzag $B'=F\tr C\tr E_1$ has type $(0,-1,-a)$. The morphisms $\eta_{wP}:W\rightarrow \mathbb{F}_1$ and $\varepsilon_{P,Q}:Y\rightarrow X$ descend respectively to birational morphisms $\eta_{wP}':X'\rightarrow \mathbb{F}_1$ and $\varepsilon_{P,Q}':X\rightarrow X'$ for which the following diagram is commutative \[\xymatrix{Y \ar[r]^{\varepsilon_{P,Q}} \ar[d]_{\mu_{P,Q}} & W \ar[d]_{\mu_{wP}} \ar[dr]^{\eta_{wP}} \\ X \ar[r]^{\varepsilon_{P,Q}'} & X' \ar[r]^{\eta_{wP}'} & \mathbb{F}_1.}\]
With the choice of coordinates made in \ref{SetUpF1}, the affine surface $S'=X'\setminus B'$ embeds into $\mathbb{A}^3=\mathrm{Spec}(k[x,y,u])$ as the subvariety defined by the equation $yu=xP(x)$ in such a way that the restriction of $\eta_{wP}'$ to it coincides with the projection ${\rm pr}_{x,y}\mid_{S'}:S'\rightarrow \mathbb{A}^2\subset \mathbb{F}_1$ (see e.g. \cite[Lemma 5.4.4]{first}). One checks further that $S=X\setminus B$ can be embeded into $\mathbb{A}^4=\mathrm{Spec}(\k[x,y,u,v])$ as the subvariety $S$ given by the following system of equations  
$$\left\{\begin{array}{lcl}
yu&=&xP(x)\\
xv&=&uQ(u)\\
yv&=&P(x)Q(u),\end{array}\right.$$
so that $\varepsilon_{P,Q}':X\rightarrow X'$ restricts on $S$ to the projection ${\rm pr}_{x,y,u}\mid_S: S\rightarrow S'\subset \mathbb{A}^3$. In this description, the intersection with $S$ of the  irreducible components $A_i$ and $B_j$ of $F_0$ coincide respectively with the irreducible components $\{y=x-\alpha_i=0\}$, $i=1,\ldots,l$, and $\{y=x=u-\beta_j\}$, $j=1,\ldots,m$, of the degenerate fiber of the induced $\mathbb{A}^1$-fibration $\bar{\pi}\mid_S=\mathrm{pr}_y:S\rightarrow \mathbb{A}^1$ .
\end{enavant}

\subsection{Isomorphism classes}

\indent\newline\noindent Here we show that the construction of the previous subsection describes all possible isomorphism types of normal affine $\mathbb{A}^1$-fibered surfaces admitting a completion into a standard pair of type $(0,-1,-a,-b)$. We characterize their isomorphism classes in terms of the corresponding polynomials $P$ and $Q$.  

\begin{prop}\label{Lem:ModelP1P2}
Let $(X,B=F\tr C\tr E_1\tr E_2,\overline{\pi})$ be a standard pair of type $(0,-1,-a,-b)$, $a,b\geq 2$, with a minimal resolution of singularities $\mu:\left(Y,B,\overline{\pi}\circ\mu\right)\rightarrow\left(X,B,\overline{\pi}\right)$ and let $\eta:Y\rightarrow \mathbb{F}_1$ be the birational morphism as in $\S \ref{SetUpF1}$ above.  
If $X\setminus B$ is affine then the morphisms $\eta$, $\mu$ are equal to that $\eta_{P,Q}$, $\mu_{P,Q}$ defined in $\S\ref{PairConstruction}$, for some polynomials $P,Q\in \k[w]$ of degree~$a-1$ and $b-1$ respectively. 

In particular, every normal affine surface completable by a zigzag of type $(0,-1,-a,-b)$ $(a,b\geq 2)$ is isomorphic to a one in $\mathbb{A}^4=\mathrm{Spec}(\k[x,y,u,v])$ defined by a system of equations of the form
$$\left\{\begin{array}{lcl}
yu&=&xP(x)\\
vx&=&uQ(u)\\
yv&=&P(x)Q(u).\end{array}\right .$$
\end{prop}
\begin{proof} Since $\eta$ maps $C$ to the $(-1)$-curve of $\mathbb{F}_1$, it follows that $E_1$ is the strict transform of $L\subset \mathbb{F}_1$. We may factor $\eta$ as $\eta_1\circ \eta_2$, where $\eta_1:Y'\rightarrow \mathbb{F}_1$ is the minimal blow-up which extracts $E_2$ and $\eta_2:Y\rightarrow Y'$ is another birational morphism. By definition $\eta_1$ is the blow-up of a sequence of points $p_1,...,p_n$  such that for every $i=2,\ldots , n$,  $p_i$ is in the first neighborhood of $p_{i-1}$ and such that $E_2$ is the exceptional divisor of the the blow-up of $p_n$. The fact that $E_2$ and $E_1$ intersect each other implies that $p_n$ and hence each $p_i$ belong to the strict transform of $L_0=L\setminus C$. Since $S=X\setminus B$ is affine, it follows from \cite[Lemma 1.4.2 p. 195]{MiyBook} that every $(-1)$-curve in the degenerate fibre of $\overline{\pi}\circ \mu$ intersects either $E_1$ or $E_2$. This implies that there exist points $\alpha_1,\dots,\alpha_l,\beta_1,\dots,\beta_k\in Y'$ where each $\alpha_i$ belongs to $E_1\setminus (E_2\cup C)\cong \kk^{*}$, each  $\beta_i$ belongs to $E_2\setminus E_1\cong \kk$, and some multiplicities associated to them, so that $\eta_2$ is the blow-up of the points $\alpha_i$ and $\beta_i$ and of infinitely near points belonging only to $E_1$ and $E_2$ respectively. Taking an appropriate  parametrisation for the $\alpha_i$'s in $E_1 \setminus (C \cup E_2)\cong \kk^{*}$ and the $\beta_i$'s in $E_2\setminus E_1\cong \kk$ yields the polynomials $P$ and $Q$ respectively.
\end{proof}

In general (i.e.\ for pairs of type $(0,-1,-a_1,\dots,-a_r)$ with large $r$), it may happen that non isomorphic standard pairs $(X,B,\overline{\pi})$ and $(X',B',\overline{\pi}')$  give rise to isomorphic $\mathbb{A}^1$-fibered quasi-projective surfaces $(X\setminus B,\pi|_{X\setminus B})$ and $(X'\setminus B',\pi'|_{X'\setminus B'})$. However, the following proposition shows in particular that the $\mathbb{A}^1$-fibered affine surfaces $\pi:S\rightarrow \mathbb{A}^1$ considered above all admit a unique compatible projective model.

\begin{prop}\label{Prop:IsoPairs}
Let $(X,B,\overline{\pi})$ and $(X',B',\overline{\pi}')$ be two standard pairs of type $(0,-1,-a,-b)$ obtained 
from pairs of polynomials $(P,Q)$ and $(P',Q')$ via the construction of $\S\ref{PairConstruction}$.

\begin{enumerate}
\item
The pairs $(X,B)$ and $(X',B')$ are isomorphic if and only if the $\mathbb{A}^1$-fibered surfaces $(X\setminus B,\overline{\pi}|_{X\setminus B})$ and $(X'\setminus B',\overline{\pi}'|_{X'\setminus B'})$ are isomorphic. 
\item
This is the case if and only if one of the following holds:
\begin{enumerate}
\item
$P(0)P'(0)\not=0$ and  $P'(w)=\alpha P(\beta w)$, $Q'(w)=\gamma Q(\delta w+t)$ for some $\alpha,\beta,\gamma,\delta \in \k^{*}$, $t\in \k$.
\item
$P(0)=P'(0)=0$ and $P'(w)=\alpha P(\beta w)$, $Q'(w)=\gamma Q(\delta w)$ for some $\alpha,\beta,\gamma,\delta \in \k^{*}$.
\end{enumerate}
\item
Letting $r_0$ be the multiplicity of $0$ in $P$, the automorphism group $\Aut(X,B)$ of the pair $(X,B)$ consists of lifts of automorphisms of $\mathbb{A}^2\subset \mathbb{F}_1$ of the form
\begin{center}
\begin{tabular}{ll}
$\{(x,y)\mapsto (ax+by,cy)\ |\  P(aw)/P(w)\in \k^{*}, Q(\frac{aw-b}{c})/Q(w)\in \k^{*}\}$& if $r_0=0$,\\
$\{(x,y)\mapsto (ax+by,cy)\ |\  P(aw)/P(w)\in \k^{*}, Q(\frac{a^{r_0+1}}{c}\cdot w)/Q(w)\in \k^{*}\}$& if $r_0\ge 1$.
\end{tabular}\end{center}
\end{enumerate}
\end{prop}
\begin{proof}
Let $\mu_{P,Q}:Y\rightarrow X$ and $\eta_{P,Q}:Y\rightarrow \mathbb{F}_1$ be the morphisms defined in \S \ref{PairConstruction}, and the same with primes. By virtue of \cite[Lemma 5.2.1 or Lemma 2.2.3]{first}, $(X\setminus B,\overline{\pi}|_{X\setminus B})$ and $(X'\setminus B',\overline{\pi'}|_{X'\setminus B'})$ are isomorphic if and only if there exists an automorphism $\psi$ of $\mathbb{A}^2\subset \mathbb{F}_1$ preserving the $\mathbb{A}^1$-fibration $\mathrm{pr}_y$ and sending the base locus $Z$ of $\eta_{P,Q}^{-1}$ isomorphically onto that $Z'$ of $\eta_{P',Q'}^{-1}$ while $(X,B,\overline{\pi})$ and $(X',B',\overline{\pi}')$ are isomorphic if and only if there exists an affine automorphism $\psi$ of this type. An automorphism $f$ preserving the fibration $\mathrm{pr}_y$ and mapping $Z$ isomorphically onto $Z'$ must preserve the fiber $L_0=\mathrm{pr}_y^{-1}(0)$ and fix the point $(0,0)$. 
Thus $f$ has the form $f \colon (x,y)\mapsto (a x+ yR(y),c y)$, with $a,c\in \k^{*}$ and $R\in \k[y]$. Such an automorphism acts on $L_0$ by $x\mapsto a x$. We identify points of $E_2\setminus L$ with directions $u+\beta v=0$ in the blow-up $(u,v)\mapsto (u,u^{r_0}v)$ as in $\S\ref{PairConstruction}$,where $r_0\ge 0$ denotes the multiplicity of $0$ as a root of $P$. We claim that $f$ acts on $E_2$ in the following way:  
$$\left\{\begin{array}{rcl}
\beta \mapsto  \frac{a\beta-R(0)}{c}& \mbox{ if }r_0=0\\
\beta\mapsto \frac{a^{r_0+1}}{c}\cdot \beta & \mbox{ if }r_0\ge 1\\
\end{array}\right.$$
Indeed, if $r_0=0$, then the action of $f^{-1}$  on the tangent directions is given by $u+\beta v \mapsto au+vR(0)+\beta cv=a(u+\frac{R(0)+\beta c}{a}v)$ and so $f$ maps $\beta$ to $(a\beta-R(0))/c$. Otherwise, if $r_0>0$, then the lift of $f$ by $(u,v)\dasharrow (u,u^{r_0}v)$ takes the form 
\begin{center}$(u,v)\mapsto  (au+u^{r_0}vR(u^{r_0}v), \frac{cv}{(a+u^{r_0-1}vR(u^{r_0}v))^{r_0}}).$\end{center}
In the local chart $(\hat{u},{v})\mapsto (\hat{u}v,v)=(u,v)$ of the blow-up of the origin $(0,0)$, the later lifts further to the map 
\begin{center}$(\hat{u},v)\mapsto (\frac{(a\hat{u}+(\hat{u}v)^{r_0}R((\hat{u}v)^{r_0}v))(a+(\hat{u}v)^{r_0-1}vR((\hat{u}v)^{r_0}v))^{r_0}}{c}, \frac{cv}{(a+(\hat{u}v)^{r_0-1}vR((\hat{u}v)^{r_0}v))^{r_0}}).$\end{center}
By construction the tangent direction $u+\beta v=0$ corresponds to the point $(\hat{u},v)=(-\beta,0)$ which is thus mapped to $(-\beta \cdot \frac{a^{r_0+1}}{c},0)$ as claimed.

It follows from the above description that the affine automorphism $\psi\colon (x,y)\mapsto(a x+ yR(0),cy)$ also maps $Z$ isomorphically onto $Z'$ and so, we obtain the equivalence between isomorphism classes of standard pairs and isomorphism classes between induced $\mathbb{A}^1$-fibered surfaces. The second assertion then follows immediately from the description of the action of $\psi$ on $L_0$ and $E_2$.  

Finally, as explained earlier, the group $\Aut(X,B)$ consists of lifts of automorphisms of $\mathbb{F}_1$ which preserve the set $Z$. Since they fix the origin $(0,0)$, these automorphism can be written in the form $(x,y)\mapsto (a x+by ,cy)$ where $a,c\in \k^{*}$, $b\in \k$. By virtue of the above description, the induced action  on the line $L_0=\mathrm{pr}_y^{-1}(0)$ which supports the points of $Z$ corresponding to roots of $P$  is given by $x\mapsto a x$, whereas the action on the line $E_2\setminus L$ which supports the points of $Z$ corresponding to roots of $Q$ is either $\beta \mapsto \frac{a\beta-b}{c}$, or $\beta\mapsto \frac{a^{r_0+1}}{c}\cdot \beta$, depending if $r_0$ is $0$ or positive. This yields the last assertion.
\end{proof}
\begin{cor}\label{Cor:IsoSameType}
Let $(X,B,\overline{\pi})$ and $(X',B',\overline{\pi}')$ be two standard pairs of type $(0,-1,-a,-b)$ obtained 
from pairs of polynomials $(P,Q)$ and $(P',Q')$ via the construction of $\S\ref{PairConstruction}$. If  the pairs $(X,B)$ and $(X',B')$ are isomorphic, then one of the following holds:

\begin{enumerate}
\item
Both pairs are of type $\mathrm{I}$, i.e.\ $P(0)P'(0)\not=0$;
\item
Both pairs are of type $\mathrm{II}$, i.e.\ $P(0)=P'(0)=0$ and $Q(0)Q'(0)\not=0$;
\item
Both pairs are of type $\mathrm{III}$, i.e.\ $P(0)=P'(0)=Q(0)=Q'(0)=0$.
\end{enumerate}
\end{cor}
\begin{proof}
Follows directly from Proposition~\ref{Prop:IsoPairs}.
\end{proof}

\section{Reversions between standard pairs of type $(0,-1,-a,-b)$}

  To classify the existing $\mathbb{A}^1$-fibrations on the affine surfaces constructed in the previous section, the next step consists in studying birational maps between the corresponding standard pairs $(X,B)$. In view of the description recalled in \S\ref{links-recolec}, this amounts to describe all possible fibered modifications and reversions between these pairs. Since Proposition \ref{Prop:IsoPairs} guarantees that there cannot exists fibered modifications between two non isomorphic such pairs, it remains to characterize the possible reversions between these.

\subsection{Preliminaries}
\indent\newline\noindent Here we set-up notations that will be used in the sequel to describe the geometry of the different pairs that can obtained by reversing a given standard pair $(X,B=F\tr C \tr E)$ of type $(0,-1,-a,-b)$.
\begin{enavant}\label{DescriptionReversion2} For such pairs, the general description of reversions given in \ref{DescriptionReversion1} above specializes to the following simpler form:  
Given a $\k$-rational point $p\in F\setminus C$, the contraction of $C$ followed by the blow-up of $p$ yields a birational map $\theta_0:(X,B)\dasharrow (X_0,B_0)$ to a pair with a zigzag of type $(-b,-a+1,0,-1)$. Then we produce a birational map $\varphi_1:(X_0,B_0)\dasharrow (X_1',B_1')$, where $B_1'$ is of type $(-b,-1,0,-a+1)$. The blow-down of the $(-1)$-curve in $B_1'$ followed by the blow-up of the point of intersection of its $(0)$-curve with the curve immediately after it yields a birational map $\theta_1:(X_1',B_1')\dasharrow (X_1,B_1)$ where $B_1$ is a zigzag of type $(-b+1,0,-1,-a)$. Repeating this process yields birational maps $\theta_0,\varphi_1,\theta_1,\varphi_2,\theta_2$ described by the following figure.
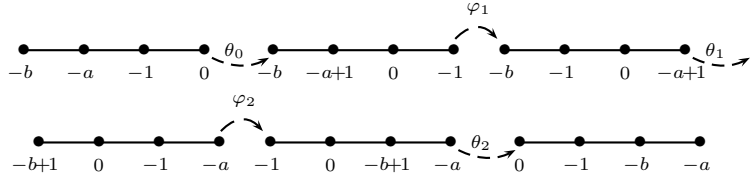
\begin{center}\begin{tabular}{l}
\begin{pspicture}(1,0.4)(4,1.8)
\psline(0.8,1)(3.2,1)
\rput(0.8,1){\textbullet}\rput(0.75,0.7){{\scriptsize $-b$}}
\rput(1.6,1){\textbullet}\rput(1.55,0.7){{\scriptsize $-a$}}
\rput(2.4,1){\textbullet}\rput(2.35,0.7){{\scriptsize $-1$}}
\rput(3.2,1){\textbullet}\rput(3.2,0.7){{\scriptsize $0$}}
\rput(3.6,1){{\scriptsize $\theta_0$}}
\parabola[linestyle=dashed]{->}(3.3,0.9)(3.68,0.8)
\end{pspicture}
\begin{pspicture}(0.8,0.4)(3.76,1.5)
\psline(0.8,1)(3.2,1)
\rput(0.8,1){\textbullet}\rput(0.75,0.7){{\scriptsize $- b$}}
\rput(1.6,1){\textbullet}\rput(1.55,0.7){{\scriptsize $- a\!\!+\!\!1$}}
\rput(2.4,1){\textbullet}\rput(2.4,0.7){{\scriptsize $0$}}
\rput(3.2,1){\textbullet}\rput(3.15,0.7){{\scriptsize $-1$}}
\rput(3.53,1.55){{\scriptsize $\varphi_1$}}
\parabola[linestyle=dashed]{->}(3.22,1.15)(3.5,1.3)
\end{pspicture}
\begin{pspicture}(0.8,0.4)(4,1.5)
\psline(0.8,1)(3.2,1)
\rput(0.8,1){\textbullet}\rput(0.75,0.7){{\scriptsize $- b$}}
\rput(1.6,1){\textbullet}\rput(1.55,0.7){{\scriptsize $-1$}}
\rput(2.4,1){\textbullet}\rput(2.4,0.7){{\scriptsize $0$}}
\rput(3.2,1){\textbullet}\rput(3.15,0.7){{\scriptsize $- a\!\!+\!\! 1$}}
\rput(3.6,1){{\scriptsize $\theta_1$}}\parabola[linestyle=dashed]{->}(3.3,0.9)(3.68,0.8)\end{pspicture}\\
\begin{pspicture}(0.8,0.4)(3.76,1.5)
\psline(0.8,1)(3.2,1)
\rput(0.8,1){\textbullet}\rput(0.75,0.7){{\scriptsize $- b\!\!+\!\! 1$}}
\rput(1.6,1){\textbullet}\rput(1.6,0.7){{\scriptsize $0$}}
\rput(2.4,1){\textbullet}\rput(2.35,0.7){{\scriptsize $-1$}}
\rput(3.2,1){\textbullet}\rput(3.15,0.7){{\scriptsize $- a$}}
\rput(3.53,1.55){{\scriptsize $\varphi_2$}}
\parabola[linestyle=dashed]{->}(3.22,1.15)(3.5,1.3)\end{pspicture}
\begin{pspicture}(0,0.4)(4,1.5)
\psline(0,1)(1.6,1)
\psline(1.6,1)(2.4,1)
\rput(0,1){\textbullet}\rput(-0.05,0.7){{\scriptsize $-1$}}
\rput(0.8,1){\textbullet}\rput(0.8,0.7){{\scriptsize $0$}}
\rput(1.6,1){\textbullet}\rput(1.55,0.7){{\scriptsize $-b\!\!+\!\! 1$}}
\rput(2.4,1){\textbullet}\rput(2.35,0.7){{\scriptsize $- a$}}
\rput(2.8,1){{\scriptsize $\theta_2$}}\parabola[linestyle=dashed]{->}(2.5,0.9)(2.88,0.8)\end{pspicture}
\begin{pspicture}(0.8,0.4)(3.5,1.5)
\psline(0,1)(1.6,1)
\psline(1.6,1)(2.4,1)
\rput(0,1){\textbullet}\rput(0,0.7){{\scriptsize $0$}}
\rput(0.8,1){\textbullet}\rput(0.75,0.7){{\scriptsize $-1$}}
\rput(1.6,1){\textbullet}\rput(1.55,0.7){{\scriptsize $-b$}}
\rput(2.4,1){\textbullet}\rput(2.35,0.7){{\scriptsize $- a$}}
\end{pspicture}
\end{tabular}\end{center}
The reversion of $(X,B)$ with center at $p\in F\setminus C$ is then the composition $$\phi=\theta_2\varphi_2\theta_1\varphi_1\theta_0:(X,B)\dasharrow(X',B')=(X_2,B_2).$$
\end{enavant}

\begin{enavant}\label{NotationAiBiAlphaiBetai} The model of the pair $(X',B')$ obtained after a reversion is essentially determined by the proper transform of the lines in $\p^2$ passing through the image of the center $p$ of the reversion by the blow-up $\tau :\mathbb{F}^1\rightarrow \p^2$. By virtue of Proposition \ref{Lem:ModelP1P2} we may assume that the initial pair $(X,B)$ is obtained from a pair of polynomial $P,Q\in \k[w]$ of respective degrees $a-1,b-1\geq1$ by means of the construction described in $\S\ref{PairConstruction}$. We let  \[\begin{array}{rcccl}(X,B=F\tr C\tr E_1\tr E_2,\overline{\pi})&\stackrel{\mu_{P,Q}}{\longleftarrow}&(Y,B,\overline{\pi}\mu_{P,Q})&\stackrel{\eta_{P,Q}}{\longrightarrow}&(\mathbb{F}_1,F\tr C\tr L)\end{array}\] be the corresponding birational morphisms, where $\eta_{P,Q}=\eta_{wP}\circ\varepsilon_{P,Q}$. The $\k$-rational point $p\in F\setminus C$ corresponds via $\tau\colon \mathbb{F}_1\to\mathbb{P}^2$ to a point $(\lambda:1:0)\in \mathbb{P}^2$ for some $\lambda \in \k$. 
For every root $\alpha_i$ of $P$ in $\kk$ we denote by $D_i\subset Y$  the proper transform by $(\tau\circ \eta_{P,Q})^{-1}$ of the line of equation $x-\lambda y-\alpha_i z=0$, passing through $(\lambda:1:0)$ and $\tau(\alpha_i)=(\alpha_i:0:1)$. Recall that if $P(0)\neq 0$ then $\beta_1,\ldots, \beta_m$ denote the distinct roots of $Q$ in $\kk$. With this notation, we have the following description. 
\end{enavant}

\begin{lem}\label{Lem:P1P2Nota} The possible dual graphs for the divisor $\eta_{P,Q}^{-1}(F\tr C \tr L)\cup \bigcup_{i=0}^m D_i$  are the following:
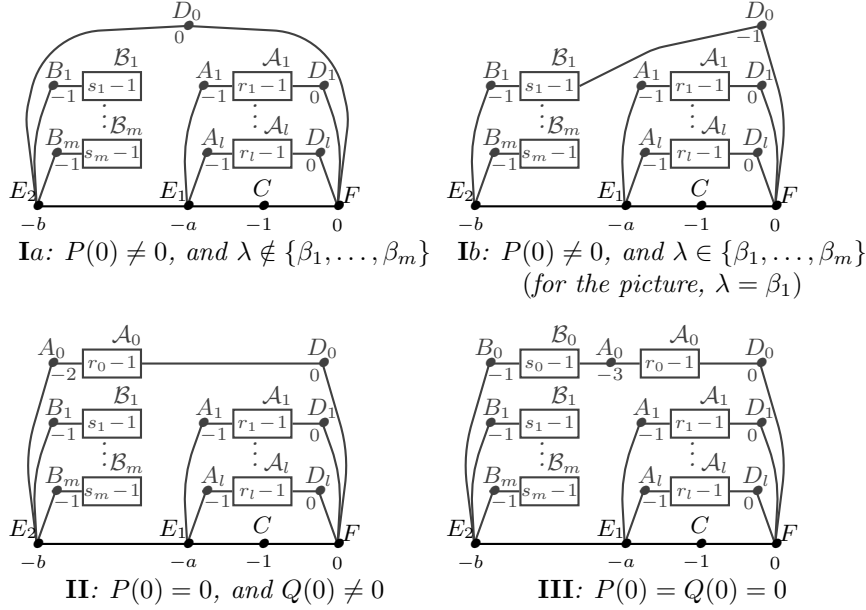
\begin{center}
\begin{tabular}{cc}
\begin{pspicture}(-1,0.7)(4,4)
{\darkgray
\psline[linecolor=darkgray](-1,1)(-0.75,1.7)\rput(-0.75,1.7){\textbullet}\rput(-0.6,1.55){\scriptsize $-1$}\rput(-0.65,1.9){\small $B_m$}
\psline[linecolor=darkgray,linearc=2](-1,1)(-1.05,1.9)(-0.8,2.6)\rput(-0.8,2.6){\textbullet}\rput(-0.65,2.45){\scriptsize $-1$}\rput(-0.7,2.8){\small $B_1$}
\psline[linecolor=darkgray](-0.8,2.6)(-0.4,2.6)
\psframe[linecolor=darkgray](-0.4,2.4)(0.4,2.8)
\psline[linecolor=darkgray](-0.8,1.7)(-0.4,1.7)
\psframe[linecolor=darkgray](-0.4,1.5)(0.4,1.9)
\rput(0,1.7){\scriptsize $s_m\!-\!1$}\rput(0.2,2.07){\small $\mathcal{B}_m$}
\rput(-0.1,2.3){$\vdots$}
\rput(0,2.6){\scriptsize $s_1\!-\!1$}\rput(0.2,2.97){\small $\mathcal{B}_1$}
\psline[linecolor=darkgray](1,1)(1.25,1.7)\rput(1.25,1.7){\textbullet}\rput(1.4,1.55){\scriptsize $-1$}\rput(1.35,1.9){\small $A_l$}
\psline[linecolor=darkgray,linearc=2](1,1)(0.95,1.9)(1.2,2.6)\rput(1.2,2.6){\textbullet}\rput(1.35,2.45){\scriptsize $-1$}\rput(1.3,2.8){\small $A_1$}
\psline[linecolor=darkgray](3,1)(2.75,1.7)\rput(2.75,1.7){\textbullet}\rput(2.65,1.55){\scriptsize $0$}\rput(2.75,1.9){\small $D_l$}
\psline[linecolor=darkgray,linearc=2](3,1)(3.05,1.9)(2.8,2.6)\rput(2.8,2.6){\textbullet}\rput(2.7,2.45){\scriptsize $0$}\rput(2.8,2.8){\small $D_1$}
\rput(1,3.4){\textbullet}\rput(0.9,3.25){\scriptsize $0$}\rput(1,3.6){\small $D_0$}
\psline[linecolor=darkgray,linearc=1](3,1)(3.15,2.2)(2.9,3.3)(1,3.4)\psline[linecolor=darkgray,linearc=1](-1,1)(-1.15,2.2)(-0.9,3.3)(1,3.4)
\psline[linecolor=darkgray](1.2,2.6)(1.6,2.6)
\psline[linecolor=darkgray](2.4,2.6)(2.8,2.6)
\psframe[linecolor=darkgray](1.6,2.4)(2.4,2.8)
\psline[linecolor=darkgray](1.2,1.7)(1.6,1.7)
\psline[linecolor=darkgray](2.4,1.7)(2.8,1.7)
\psframe[linecolor=darkgray](1.6,1.5)(2.4,1.9)
\rput(2,1.7){\scriptsize $r_l\!-\!1$}\rput(2.2,2.07){\small $\mathcal{A}_l$}
\rput(1.9,2.3){$\vdots$}
\rput(2,2.6){\scriptsize $r_1\!-\!1$}\rput(2.2,2.97){\small $\mathcal{A}_1$}
}
\psline(-1,1)(3,1)
\rput(0.82,1.2){{\small $E_1$}}
\rput(-1.18,1.2){{\small $E_2$}}
\rput(2,1.25){{\small $C$}}
\rput(3.18,1.15){{\small $F$}}
\rput(-1,1){\textbullet}\rput(-1.05,0.75){{\scriptsize $-b$}}
\rput(1,1){\textbullet}\rput(0.95,0.75){{\scriptsize $-a$}}
\rput(2,1){\textbullet}\rput(1.95,0.75){{\scriptsize $-1$}}
\rput(3,1){\textbullet}\rput(3,0.75){{\scriptsize $0$}}
\end{pspicture}
&
\begin{pspicture}(-1,0.7)(4,3.6)
{\darkgray
\psline[linecolor=darkgray](-1,1)(-0.75,1.7)\rput(-0.75,1.7){\textbullet}\rput(-0.6,1.55){\scriptsize $-1$}\rput(-0.65,1.9){\small $B_m$}
\psline[linecolor=darkgray,linearc=2](-1,1)(-1.05,1.9)(-0.8,2.6)\rput(-0.8,2.6){\textbullet}\rput(-0.65,2.45){\scriptsize $-1$}\rput(-0.7,2.8){\small $B_1$}
\psline[linecolor=darkgray](-0.8,2.6)(-0.4,2.6)
\psframe[linecolor=darkgray](-0.4,2.4)(0.4,2.8)
\psline[linecolor=darkgray](-0.8,1.7)(-0.4,1.7)
\psframe[linecolor=darkgray](-0.4,1.5)(0.4,1.9)
\rput(0,1.7){\scriptsize $s_m\!-\!1$}\rput(0.2,2.07){\small $\mathcal{B}_m$}
\rput(-0.1,2.3){$\vdots$}
\rput(0,2.6){\scriptsize $s_1\!-\!1$}\rput(0.2,2.97){\small $\mathcal{B}_1$}
\psline[linecolor=darkgray](1,1)(1.25,1.7)\rput(1.25,1.7){\textbullet}\rput(1.4,1.55){\scriptsize $-1$}\rput(1.35,1.9){\small $A_l$}
\psline[linecolor=darkgray,linearc=2](1,1)(0.95,1.9)(1.2,2.6)\rput(1.2,2.6){\textbullet}\rput(1.35,2.45){\scriptsize $-1$}\rput(1.3,2.8){\small $A_1$}
\psline[linecolor=darkgray](3,1)(2.75,1.7)\rput(2.75,1.7){\textbullet}\rput(2.65,1.55){\scriptsize $0$}\rput(2.75,1.9){\small $D_l$}
\psline[linecolor=darkgray,linearc=2](3,1)(3.05,1.9)(2.8,2.6)\rput(2.8,2.6){\textbullet}\rput(2.7,2.45){\scriptsize $0$}\rput(2.8,2.8){\small $D_1$}
\psline[linecolor=darkgray,linearc=2](3,1)(3.15,2.2)(2.8,3.4)\rput(2.8,3.4){\textbullet}\rput(2.65,3.25){\scriptsize $-1$}\rput(2.8,3.6){\small $D_0$}
\psline[linecolor=darkgray](1.2,2.6)(1.6,2.6)
\psline[linecolor=darkgray](2.4,2.6)(2.8,2.6)
\psframe[linecolor=darkgray](1.6,2.4)(2.4,2.8)
\psline[linecolor=darkgray](1.2,1.7)(1.6,1.7)
\psline[linecolor=darkgray](2.4,1.7)(2.8,1.7)
\psframe[linecolor=darkgray](1.6,1.5)(2.4,1.9)
\rput(2,1.7){\scriptsize $r_l\!-\!1$}\rput(2.2,2.07){\small $\mathcal{A}_l$}
\rput(1.9,2.3){$\vdots$}
\rput(2,2.6){\scriptsize $r_1\!-\!1$}\rput(2.2,2.97){\small $\mathcal{A}_1$}
\psline[linecolor=darkgray,linearc=1](0.4,2.6)(1.4,3.1)(2.8,3.4)
}
\psline(-1,1)(3,1)
\rput(0.82,1.2){{\small $E_1$}}
\rput(-1.18,1.2){{\small $E_2$}}
\rput(2,1.25){{\small $C$}}
\rput(3.18,1.15){{\small $F$}}
\rput(-1,1){\textbullet}\rput(-1.05,0.75){{\scriptsize $-b$}}
\rput(1,1){\textbullet}\rput(0.95,0.75){{\scriptsize $-a$}}
\rput(2,1){\textbullet}\rput(1.95,0.75){{\scriptsize $-1$}}
\rput(3,1){\textbullet}\rput(3,0.75){{\scriptsize $0$}}
\end{pspicture}\\
$\mathbf{I}a$: $P(0)\not=0$, and $\lambda\notin\{\beta_1,\dots,\beta_m\}$ & 
$\mathbf{I}b$: $P(0)\not=0$, and $\lambda\in\{\beta_1,\dots,\beta_m\}$\\ & $($for the picture, $\lambda=\beta_1)$ \\
\\
\begin{pspicture}(-1,0.7)(4,3.8)
{\darkgray
\psline[linecolor=darkgray](-1,1)(-0.75,1.7)\rput(-0.75,1.7){\textbullet}\rput(-0.6,1.55){\scriptsize $-1$}\rput(-0.65,1.9){\small $B_m$}
\psline[linecolor=darkgray,linearc=2](-1,1)(-1.05,1.9)(-0.8,2.6)\rput(-0.8,2.6){\textbullet}\rput(-0.65,2.45){\scriptsize $-1$}\rput(-0.7,2.8){\small $B_1$}\psline[linecolor=darkgray,linearc=2](-1,1)(-1.15,2.2)(-0.8,3.4)
\rput(-0.8,3.4){\textbullet}\rput(-0.65,3.25){\scriptsize $-2$}\rput(-0.8,3.6){\small $A_0$}
\psline[linecolor=darkgray](-0.8,2.6)(-0.4,2.6)
\psframe[linecolor=darkgray](-0.4,2.4)(0.4,2.8)
\psline[linecolor=darkgray](-0.8,3.4)(-0.4,3.4)
\psframe[linecolor=darkgray](-0.4,3.2)(0.4,3.6)
\psline[linecolor=darkgray](-0.8,1.7)(-0.4,1.7)
\psframe[linecolor=darkgray](-0.4,1.5)(0.4,1.9)
\rput(0,1.7){\scriptsize $s_m\!-\!1$}\rput(0.2,2.07){\small $\mathcal{B}_m$}
\rput(-0.1,2.3){$\vdots$}
\rput(0,2.6){\scriptsize $s_1\!-\!1$}\rput(0.2,2.97){\small $\mathcal{B}_1$}
\rput(0,3.4){\scriptsize $r_0\!-\!1$}\rput(0.2,3.77){\small $\mathcal{A}_0$}
\psline[linecolor=darkgray](1,1)(1.25,1.7)\rput(1.25,1.7){\textbullet}\rput(1.4,1.55){\scriptsize $-1$}\rput(1.35,1.9){\small $A_l$}
\psline[linecolor=darkgray,linearc=2](1,1)(0.95,1.9)(1.2,2.6)\rput(1.2,2.6){\textbullet}\rput(1.35,2.45){\scriptsize $-1$}\rput(1.3,2.8){\small $A_1$}
\psline[linecolor=darkgray](3,1)(2.75,1.7)\rput(2.75,1.7){\textbullet}\rput(2.65,1.55){\scriptsize $0$}\rput(2.75,1.9){\small $D_l$}
\psline[linecolor=darkgray,linearc=2](3,1)(3.05,1.9)(2.8,2.6)\rput(2.8,2.6){\textbullet}\rput(2.7,2.45){\scriptsize $0$}\rput(2.8,2.8){\small $D_1$}
\psline[linecolor=darkgray,linearc=2](3,1)(3.15,2.2)(2.8,3.4)\rput(2.8,3.4){\textbullet}\rput(2.7,3.25){\scriptsize $0$}\rput(2.8,3.6){\small $D_0$}
\psline[linecolor=darkgray](1.2,2.6)(1.6,2.6)
\psline[linecolor=darkgray](2.4,2.6)(2.8,2.6)
\psframe[linecolor=darkgray](1.6,2.4)(2.4,2.8)
\psline[linecolor=darkgray](1.2,1.7)(1.6,1.7)
\psline[linecolor=darkgray](2.4,1.7)(2.8,1.7)
\psframe[linecolor=darkgray](1.6,1.5)(2.4,1.9)
\rput(2,1.7){\scriptsize $r_l\!-\!1$}\rput(2.2,2.07){\small $\mathcal{A}_l$}
\rput(1.9,2.3){$\vdots$}
\rput(2,2.6){\scriptsize $r_1\!-\!1$}\rput(2.2,2.97){\small $\mathcal{A}_1$}
\psline[linecolor=darkgray](0.4,3.4)(2.8,3.4)
}
\psline(-1,1)(3,1)
\rput(0.82,1.2){{\small $E_1$}}
\rput(-1.18,1.2){{\small $E_2$}}
\rput(2,1.25){{\small $C$}}
\rput(3.18,1.15){{\small $F$}}
\rput(-1,1){\textbullet}\rput(-1.05,0.75){{\scriptsize $-b$}}
\rput(1,1){\textbullet}\rput(0.95,0.75){{\scriptsize $-a$}}
\rput(2,1){\textbullet}\rput(1.95,0.75){{\scriptsize $-1$}}
\rput(3,1){\textbullet}\rput(3,0.75){{\scriptsize $0$}}
\end{pspicture}
&
\begin{pspicture}(-1,0.7)(4,3.8)
{\darkgray
\psline[linecolor=darkgray](-1,1)(-0.75,1.7)\rput(-0.75,1.7){\textbullet}\rput(-0.6,1.55){\scriptsize $-1$}\rput(-0.65,1.9){\small $B_m$}
\psline[linecolor=darkgray,linearc=2](-1,1)(-1.05,1.9)(-0.8,2.6)\rput(-0.8,2.6){\textbullet}\rput(-0.65,2.45){\scriptsize $-1$}\rput(-0.7,2.8){\small $B_1$}\psline[linecolor=darkgray,linearc=2](-1,1)(-1.15,2.2)(-0.8,3.4)
\rput(-0.8,3.4){\textbullet}\rput(-0.65,3.25){\scriptsize $-1$}\rput(-0.8,3.6){\small $B_0$}
\rput(0.8,3.4){\textbullet}\rput(0.8,3.25){\scriptsize $-3$}\rput(0.8,3.6){\small $A_0$}
\psline[linecolor=darkgray](-0.8,2.6)(-0.4,2.6)
\psframe[linecolor=darkgray](-0.4,2.4)(0.4,2.8)
\psline[linecolor=darkgray](-0.8,3.4)(-0.4,3.4)
\psframe[linecolor=darkgray](-0.4,3.2)(0.4,3.6)
\psline[linecolor=darkgray](0.4,3.4)(1.2,3.4)
\psframe[linecolor=darkgray](1.2,3.2)(2,3.6)
\psline[linecolor=darkgray](-0.8,1.7)(-0.4,1.7)
\psframe[linecolor=darkgray](-0.4,1.5)(0.4,1.9)
\rput(0,1.7){\scriptsize $s_m\!-\!1$}\rput(0.2,2.07){\small $\mathcal{B}_m$}
\rput(-0.1,2.3){$\vdots$}
\rput(0,2.6){\scriptsize $s_1\!-\!1$}\rput(0.2,2.97){\small $\mathcal{B}_1$}
\rput(0,3.4){\scriptsize $s_0\!-\!1$}\rput(0.2,3.77){\small $\mathcal{B}_0$}
\rput(1.6,3.4){\scriptsize $r_0\!-\!1$}\rput(1.8,3.77){\small $\mathcal{A}_0$}
\psline[linecolor=darkgray](1,1)(1.25,1.7)\rput(1.25,1.7){\textbullet}\rput(1.4,1.55){\scriptsize $-1$}\rput(1.35,1.9){\small $A_l$}
\psline[linecolor=darkgray,linearc=2](1,1)(0.95,1.9)(1.2,2.6)\rput(1.2,2.6){\textbullet}\rput(1.35,2.45){\scriptsize $-1$}\rput(1.3,2.8){\small $A_1$}
\psline[linecolor=darkgray](3,1)(2.75,1.7)\rput(2.75,1.7){\textbullet}\rput(2.65,1.55){\scriptsize $0$}\rput(2.75,1.9){\small $D_l$}
\psline[linecolor=darkgray,linearc=2](3,1)(3.05,1.9)(2.8,2.6)\rput(2.8,2.6){\textbullet}\rput(2.7,2.45){\scriptsize $0$}\rput(2.8,2.8){\small $D_1$}
\psline[linecolor=darkgray,linearc=2](3,1)(3.15,2.2)(2.8,3.4)\rput(2.8,3.4){\textbullet}\rput(2.7,3.25){\scriptsize $0$}\rput(2.8,3.6){\small $D_0$}
\psline[linecolor=darkgray](1.2,2.6)(1.6,2.6)
\psline[linecolor=darkgray](2.4,2.6)(2.8,2.6)
\psframe[linecolor=darkgray](1.6,2.4)(2.4,2.8)
\psline[linecolor=darkgray](1.2,1.7)(1.6,1.7)
\psline[linecolor=darkgray](2.4,1.7)(2.8,1.7)
\psframe[linecolor=darkgray](1.6,1.5)(2.4,1.9)
\rput(2,1.7){\scriptsize $r_l\!-\!1$}\rput(2.2,2.07){\small $\mathcal{A}_l$}
\rput(1.9,2.3){$\vdots$}
\rput(2,2.6){\scriptsize $r_1\!-\!1$}\rput(2.2,2.97){\small $\mathcal{A}_1$}
\psline[linecolor=darkgray](2,3.4)(2.8,3.4)
}
\psline(-1,1)(3,1)
\rput(0.82,1.2){{\small $E_1$}}
\rput(-1.18,1.2){{\small $E_2$}}
\rput(2,1.25){{\small $C$}}
\rput(3.18,1.15){{\small $F$}}
\rput(-1,1){\textbullet}\rput(-1.05,0.75){{\scriptsize $-b$}}
\rput(1,1){\textbullet}\rput(0.95,0.75){{\scriptsize $-a$}}
\rput(2,1){\textbullet}\rput(1.95,0.75){{\scriptsize $-1$}}
\rput(3,1){\textbullet}\rput(3,0.75){{\scriptsize $0$}}
\end{pspicture}\\
$\mathbf{II}$: $P(0)=0$, and $Q(0)\not=0$ & 
$\mathbf{III}$: $P(0)=Q(0)=0$ \end{tabular}
\end{center}
\end{lem}
\begin{proof}
The structure of the dual graph of $\eta_{P,Q}^{-1}(F\tr C \tr L)$ has been already discussed in \ref{PairConstruction} (see Figures \ref{Fig:DoubleBlowUpr0eq0}, \ref{Fig:DoubleBlowUpr0geq1} and \ref{Fig:DoubleBlowUpr0geq2}). It remains to consider the properties of the additional curves $D_i$, $i=0,\ldots,l$. By definition, $D_i$ is the proper transform of the line  $L_i\subset \p^2$ of equation $x=\lambda y+\alpha_i z$. Since the latter does not pass through the point $(1:0:0)$ blown-up by $\tau$, its proper transform in $\mathbb{F}_1$ still have self-intersection $1$. Moreover, since $D_i$ passes through $\alpha_i$ (corresponding to $(\alpha_i:0:1)$) and not through any other $\alpha_j$, with $j\ne i$, its proper transform by $\eta_{wP}:W\rightarrow \mathbb{F}_1$ has self-intersection $0$ in $W$ and it intersects $\eta_{wP}^{-1}(F\tr C \tr L)$ transversally at a point of $A_i$ if $\alpha_i$ is a simple root of $wP(w)$ and at a point on the last component of $\mathcal{A}_i$ otherwise. The only case where a point belonging to the proper transform of $L_i$ in $W$ is blown-up by $\varepsilon_{P,Q}$ is when $i=0$, $P(0)\neq0$ and $L_0$ corresponds to a tangent direction $x=\beta_j y$ for a certain root of $Q$ in $\kk$. In this case, $D_0$ has self-intersection $-1$ in $Y$ and it intersects $\eta_{P,Q}^{-1}(F\tr C \tr L)$ transversally at a point of $B_j$ if $\beta_j$ is a simple root of $Q$ and at a point on the last component of $\mathcal{B}_j$ otherwise. This gives all diagrams pictured above (recall that $A_0=E_2$ if $P(0)\neq0$).
\end{proof}

\subsection{Classification of reversions}

\indent\newline\noindent Recall that two standard pairs  $(X,B)$ and $(X',B')$ of respective types $(0,-1,-a,-b)$ and $(0,-1,-a',-b')$ can be linked by a reversion only if $a'=b$ and $b'=a$ (see e.g. $\S\ref{DescriptionReversion2}$). To decide which types of reversions can occur between the different models of standard pairs, we may thus consider the situation that  $(X,B)$ and $(X',B')$ are obtained by means of the construction of $\S\ref{PairConstruction}$ for pairs of polynomials $(P,Q)$ and $(P',Q')$  of degrees $(a-1,b-1)$ and $(b-1,a-1)$ respectively. We let let \[\begin{array}{rcccl}(X,B=F\tr C\tr E_1\tr E_2,\overline{\pi})&\stackrel{\mu_{P,Q}}{\leftarrow}&(Y,B,\overline{\pi}\mu_{P,Q})&\stackrel{\eta_{P,Q}}{\rightarrow}&(\mathbb{F}_1,F\tr C\tr L)\\
(X',B'=F'\tr C'\tr E_1'\tr E_2',\overline{\pi}')&\stackrel{\mu_{P',Q'}}{\leftarrow}&(Y',B',\overline{\pi}'\mu_{P',Q'})&\stackrel{\eta_{P',Q'}}{\rightarrow}&(\mathbb{F}_1,F\tr C\tr L)\end{array}\] be as in $\S\ref{PairConstruction}$.
Given a reversion $\phi:(X,B)\dasharrow (X',B')$ centred at $p\in F\setminus C$, with $\phi^{-1}$ centred at $p'\in F'\setminus C'$, which correspond respectively to $(\lambda:1:0),(\lambda':1:0)\in \p^2$, we use the notation of $\S\ref{NotationAiBiAlphaiBetai}$ for $\alpha_1,\dots,\alpha_l\in\kk$, $A_i,B_i,\mathcal{A}_i,\mathcal{B}_i,D_i,\subset Y$ and the same notation with primes on $Y'$. 

\begin{lem} \label{Lem:Rev01k}With the notation above, let $\psi=(\mu_{P',Q'})^{-1}\phi\mu_{P,Q}\colon Y\dasharrow Y'$ be the lift of $\phi$ . Then one of the following three situations occurs:

$(1)$ We have case $\mathbf{I}a$ on both $Y$ and $Y'$: $P(0)P'(0)P(\lambda)P'(\lambda')\not=0$. We have $l=m'$, $m=l'$ and up to renumbering, $\psi$ sends $B_i$, $\mathcal{B}_i$, $\mathcal{A}_i$ and $D_i$ on $D_i'$, $\mathcal{A}_i'$, $\mathcal{B}'_i$ and $B_i$ respectively. Moreover, $\psi(D_0)=D_0'$ and the situation is described by the following diagram:
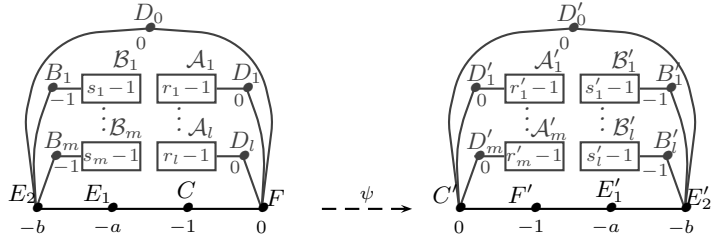
\begin{center}
\begin{pspicture}(0,0.7)(5.5,3.6)
{\darkgray
\psline[linecolor=darkgray](0,1)(0.25,1.7)\rput(0.25,1.7){\textbullet}\rput(0.4,1.55){\scriptsize $-1$}\rput(0.35,1.9){\small $B_m$}
\psline[linecolor=darkgray,linearc=2](0,1)(-.05,1.9)(0.2,2.6)\rput(0.2,2.6){\textbullet}\rput(0.35,2.45){\scriptsize $-1$}\rput(0.3,2.8){\small $B_1$}
\psline[linecolor=darkgray](0.2,2.6)(0.6,2.6)
\psframe[linecolor=darkgray](0.6,2.4)(1.4,2.8)
\psline[linecolor=darkgray](0.2,1.7)(0.6,1.7)
\psframe[linecolor=darkgray](0.6,1.5)(1.4,1.9)
\rput(1,1.7){\scriptsize $s_m\!-\!1$}\rput(1.2,2.07){\small $\mathcal{B}_m$}
\rput(0.9,2.3){$\vdots$}
\rput(1,2.6){\scriptsize $s_1\!-\!1$}\rput(1.2,2.97){\small $\mathcal{B}_1$}
\psline[linecolor=darkgray](3,1)(2.75,1.7)\rput(2.75,1.7){\textbullet}\rput(2.65,1.55){\scriptsize $0$}\rput(2.75,1.9){\small $D_l$}
\psline[linecolor=darkgray,linearc=2](3,1)(3.05,1.9)(2.8,2.6)\rput(2.8,2.6){\textbullet}\rput(2.7,2.45){\scriptsize $0$}\rput(2.8,2.8){\small $D_1$}
\psline[linecolor=darkgray,linearc=1](3,1)(3.15,2.2)(2.9,3.3)(1.5,3.4)\psline[linecolor=darkgray,linearc=1](0,1)(-0.15,2.2)(0.1,3.3)(1.5,3.4)
\rput(1.5,3.4){\textbullet}\rput(1.4,3.2){\scriptsize $0$}\rput(1.5,3.6){\small $D_0$}
\psline[linecolor=darkgray](2.4,2.6)(2.8,2.6)
\psframe[linecolor=darkgray](1.6,2.4)(2.4,2.8)
\psline[linecolor=darkgray](2.4,1.7)(2.8,1.7)
\psframe[linecolor=darkgray](1.6,1.5)(2.4,1.9)
\rput(2,1.7){\scriptsize $r_l\!-\!1$}\rput(2.2,2.07){\small $\mathcal{A}_l$}
\rput(1.9,2.3){$\vdots$}
\rput(2,2.6){\scriptsize $r_1\!-\!1$}\rput(2.2,2.97){\small $\mathcal{A}_1$}
}
\psline(0,1)(3,1)
\rput(0.82,1.2){{\small $E_1$}}
\rput(-0.18,1.2){{\small $E_2$}}
\rput(2,1.25){{\small $C$}}
\rput(3.18,1.15){{\small $F$}}
\rput(0,1){\textbullet}\rput(-0.05,0.75){{\scriptsize $-b$}}
\rput(1,1){\textbullet}\rput(0.95,0.75){{\scriptsize $-a$}}
\rput(2,1){\textbullet}\rput(1.95,0.75){{\scriptsize $-1$}}
\rput(3,1){\textbullet}\rput(3,0.75){{\scriptsize $0$}}
\rput(4.4,1.2){{\scriptsize $\psi$}}
\psline[linestyle=dashed]{->}(3.8,1)(5,1)
\end{pspicture}
\begin{pspicture}(0,0.7)(4,3.6)
{\darkgray
\psline[linecolor=darkgray](0,1)(0.25,1.7)\rput(0.25,1.7){\textbullet}\rput(0.35,1.55){\scriptsize $0$}\rput(0.35,1.9){\small $D_m'$}
\psline[linecolor=darkgray,linearc=2](0,1)(-.05,1.9)(0.2,2.6)\rput(0.2,2.6){\textbullet}\rput(0.3,2.45){\scriptsize $0$}\rput(0.3,2.8){\small $D_1'$}\psline[linecolor=darkgray,linearc=1](3,1)(3.15,2.2)(2.9,3.3)(1.5,3.4)\psline[linecolor=darkgray,linearc=1](0,1)(-0.15,2.2)(0.1,3.3)(1.5,3.4)
\rput(1.5,3.4){\textbullet}\rput(1.4,3.2){\scriptsize $0$}\rput(1.5,3.6){\small $D_0'$}
\psline[linecolor=darkgray](0.2,2.6)(0.6,2.6)
\psframe[linecolor=darkgray](0.6,2.4)(1.4,2.8)
\psline[linecolor=darkgray](0.2,1.7)(0.6,1.7)
\psframe[linecolor=darkgray](0.6,1.5)(1.4,1.9)
\rput(1,1.7){\scriptsize $r_m'\!-\!1$}\rput(1.2,2.07){\small $\mathcal{A}_m'$}
\rput(0.9,2.3){$\vdots$}
\rput(1,2.6){\scriptsize $r_1'\!-\!1$}\rput(1.2,2.97){\small $\mathcal{A}_1'$}
\psline[linecolor=darkgray](3,1)(2.75,1.7)\rput(2.75,1.7){\textbullet}\rput(2.6,1.5){\scriptsize $-1$}\rput(2.75,1.9){\small $B_l'$}
\psline[linecolor=darkgray,linearc=2](3,1)(3.05,1.9)(2.8,2.6)\rput(2.8,2.6){\textbullet}\rput(2.6,2.45){\scriptsize $-1$}\rput(2.8,2.8){\small $B_1'$}
\psline[linecolor=darkgray](2.4,2.6)(2.8,2.6)
\psframe[linecolor=darkgray](1.6,2.4)(2.4,2.8)
\psline[linecolor=darkgray](2.4,1.7)(2.8,1.7)
\psframe[linecolor=darkgray](1.6,1.5)(2.4,1.9)
\rput(2,1.7){\scriptsize $s_l'\!-\!1$}\rput(2.2,2.07){\small $\mathcal{B}_l'$}
\rput(1.9,2.3){$\vdots$}
\rput(2,2.6){\scriptsize $s_1'\!-\!1$}\rput(2.2,2.97){\small $\mathcal{B}_1'$}
}
\psline(0,1)(3,1)
\rput(0.82,1.2){{\small $F'$}}
\rput(-0.18,1.2){{\small $C'$}}
\rput(2,1.25){{\small $E_1'$}}
\rput(3.18,1.15){{\small $E_2'$}}
\rput(0,1){\textbullet}\rput(0,0.75){{\scriptsize $0$}}
\rput(1,1){\textbullet}\rput(0.95,0.75){{\scriptsize $-1$}}
\rput(2,1){\textbullet}\rput(1.95,0.75){{\scriptsize $-a$}}
\rput(3,1){\textbullet}\rput(2.95,0.75){{\scriptsize $-b$}}
\end{pspicture}
\end{center}
Furthermore, $P'(w)=cQ(dw+\lambda)$, and $P(w)=eQ'(fw+\lambda')$ for some $c,d,e,f\in \k^{*}$.

$(2)$ Up to an exchange of $Y$ and $Y'$, we have case $\mathbf{I}b$ on $Y$ and case $\mathbf{II}$ on $Y'$: $P(0)Q(0)Q'(0)\not=0$, $P'(0)=P(\lambda)=0$. Up to renumbering, $\lambda=\beta_m$, $\psi$ sends $B_m,\mathcal{B}_m$ and $D_0$ onto $D_0',\mathcal{A}_0'$ and $A_0'$ respectively, and sends the other $B_i$, $\mathcal{B}_i$, $\mathcal{A}_i$ and $D_i$ onto  $D_i'$, $\mathcal{A}_i'$, $\mathcal{B}'_i$ and $B_i$ respectively. The situation is described by the following diagram:

\medskip

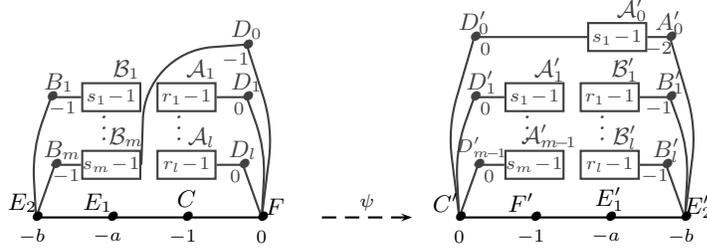
\begin{center}
\begin{pspicture}(0,0.7)(5.5,3.6)
{\darkgray
\psline[linecolor=darkgray](0,1)(0.25,1.7)\rput(0.25,1.7){\textbullet}\rput(0.4,1.55){\scriptsize $-1$}\rput(0.35,1.9){\small $B_m$}
\psline[linecolor=darkgray,linearc=2](0,1)(-.05,1.9)(0.2,2.6)\rput(0.2,2.6){\textbullet}\rput(0.35,2.45){\scriptsize $-1$}\rput(0.3,2.8){\small $B_1$}
\psline[linecolor=darkgray](0.2,2.6)(0.6,2.6)
\psframe[linecolor=darkgray](0.6,2.4)(1.4,2.8)
\psline[linecolor=darkgray](0.2,1.7)(0.6,1.7)
\psframe[linecolor=darkgray](0.6,1.5)(1.4,1.9)
\rput(1,1.7){\scriptsize $s_m\!-\!1$}\rput(1.2,2.07){\small $\mathcal{B}_m$}
\rput(0.9,2.3){$\vdots$}
\rput(1,2.6){\scriptsize $s_1\!-\!1$}\rput(1.2,2.97){\small $\mathcal{B}_1$}
\psline[linecolor=darkgray](3,1)(2.75,1.7)\rput(2.75,1.7){\textbullet}\rput(2.65,1.55){\scriptsize $0$}\rput(2.75,1.9){\small $D_l$}
\psline[linecolor=darkgray,linearc=2](3,1)(3.05,1.9)(2.8,2.6)\rput(2.8,2.6){\textbullet}\rput(2.7,2.45){\scriptsize $0$}\rput(2.8,2.8){\small $D_1$}
\psline[linecolor=darkgray,linearc=2](3,1)(3.15,2.2)(2.8,3.3)
\rput(2.8,3.3){\textbullet}\rput(2.65,3.1){\scriptsize $-1$}\rput(2.8,3.5){\small $D_0$}
\psline[linecolor=darkgray](2.4,2.6)(2.8,2.6)
\psframe[linecolor=darkgray](1.6,2.4)(2.4,2.8)
\psline[linecolor=darkgray](2.4,1.7)(2.8,1.7)
\psframe[linecolor=darkgray](1.6,1.5)(2.4,1.9)
\rput(2,1.7){\scriptsize $r_l\!-\!1$}\rput(2.2,2.07){\small $\mathcal{A}_l$}
\rput(1.9,2.3){$\vdots$}
\rput(2,2.6){\scriptsize $r_1\!-\!1$}\rput(2.2,2.97){\small $\mathcal{A}_1$}
\psline[linecolor=darkgray,linearc=1](1.4,1.7)(1.45,3.2)(2.8,3.3)
}
\psline(0,1)(3,1)
\rput(0.82,1.2){{\small $E_1$}}
\rput(-0.18,1.2){{\small $E_2$}}
\rput(2,1.25){{\small $C$}}
\rput(3.18,1.15){{\small $F$}}
\rput(0,1){\textbullet}\rput(-0.05,0.75){{\scriptsize $-b$}}
\rput(1,1){\textbullet}\rput(0.95,0.75){{\scriptsize $-a$}}
\rput(2,1){\textbullet}\rput(1.95,0.75){{\scriptsize $-1$}}
\rput(3,1){\textbullet}\rput(3,0.75){{\scriptsize $0$}}
\rput(4.4,1.2){{\scriptsize $\psi$}}
\psline[linestyle=dashed]{->}(3.8,1)(5,1)
\end{pspicture}
\begin{pspicture}(0,0.7)(4,3.6)
{\darkgray
\psline[linecolor=darkgray](0,1)(0.25,1.7)\rput(0.25,1.7){\textbullet}\rput(0.4,1.55){\scriptsize $0$}\rput(0.3,1.92){\scriptsize $D_{m\!-\!1}'$}
\psline[linecolor=darkgray,linearc=2](0,1)(-.05,1.9)(0.2,2.6)\rput(0.2,2.6){\textbullet}\rput(0.35,2.45){\scriptsize $0$}\rput(0.3,2.8){\small $D_1'$}\psline[linecolor=darkgray,linearc=2](0,1)(-0.15,2.2)(0.2,3.4)
\rput(0.2,3.4){\textbullet}\rput(0.35,3.25){\scriptsize $0$}\rput(0.2,3.6){\small $D_0'$}
\psline[linecolor=darkgray](0.2,2.6)(0.6,2.6)
\psframe[linecolor=darkgray](0.6,2.4)(1.4,2.8)
\psline[linecolor=darkgray](0.2,1.7)(0.6,1.7)
\psframe[linecolor=darkgray](0.6,1.5)(1.4,1.9)
\psline[linecolor=darkgray](0.2,3.4)(1.7,3.4)
\psframe[linecolor=darkgray](1.7,3.2)(2.5,3.6)
\psline[linecolor=darkgray](2.5,3.4)(2.8,3.4)
\rput(1,1.7){\scriptsize $s_m\!-\!1$}\rput(1.2,2.07){\small $\mathcal{A}_{m\!-\!1}'$}
\rput(0.9,2.3){$\vdots$}
\rput(1,2.6){\scriptsize $s_1\!-\!1$}\rput(1.2,2.97){\small $\mathcal{A}_1'$}
\rput(2.1,3.4){\scriptsize $s_1\!-\!1$}\rput(2.3,3.77){\small $\mathcal{A}_0'$}
\psline[linecolor=darkgray](3,1)(2.75,1.7)\rput(2.75,1.7){\textbullet}\rput(2.6,1.55){\scriptsize $-1$}\rput(2.75,1.9){\small $B_l'$}
\psline[linecolor=darkgray,linearc=2](3,1)(3.05,1.9)(2.8,2.6)\rput(2.8,2.6){\textbullet}\rput(2.65,2.45){\scriptsize $-1$}\rput(2.8,2.8){\small $B_1'$}
\psline[linecolor=darkgray,linearc=2](3,1)(3.15,2.2)(2.8,3.4)
\rput(2.8,3.4){\textbullet}\rput(2.65,3.25){\scriptsize $-2$}\rput(2.8,3.6){\small $A_0'$}
\psline[linecolor=darkgray](2.4,2.6)(2.8,2.6)
\psframe[linecolor=darkgray](1.6,2.4)(2.4,2.8)
\psline[linecolor=darkgray](2.4,1.7)(2.8,1.7)
\psframe[linecolor=darkgray](1.6,1.5)(2.4,1.9)
\rput(2,1.7){\scriptsize $r_l\!-\!1$}\rput(2.2,2.07){\small $\mathcal{B}_l'$}
\rput(1.9,2.3){$\vdots$}
\rput(2,2.6){\scriptsize $r_1\!-\!1$}\rput(2.2,2.97){\small $\mathcal{B}_1'$}
}
\psline(0,1)(3,1)
\rput(0.82,1.2){{\small $F'$}}
\rput(-0.18,1.2){{\small $C'$}}
\rput(2,1.25){{\small $E_1'$}}
\rput(3.18,1.15){{\small $E_2'$}}
\rput(0,1){\textbullet}\rput(0,0.75){{\scriptsize $0$}}
\rput(1,1){\textbullet}\rput(0.95,0.75){{\scriptsize $-1$}}
\rput(2,1){\textbullet}\rput(1.95,0.75){{\scriptsize $-a$}}
\rput(3,1){\textbullet}\rput(2.95,0.75){{\scriptsize $-b$}}
\end{pspicture}
\end{center}

Furthermore, $P'(x)=cQ(dx+\lambda)$, and $P(x)=eQ'(fx)$ for some $c,d,e,f\in \k^{*}$.

$(3)$ We have case $\mathbf{III}$ on $Y$ and $Y'$: $P(0)=P'(0)=Q(0)=Q'(0)=0$. Up to renumbering $\psi$ sends $B_i$, $\mathcal{B}_i$, $\mathcal{A}_i$ and $D_i$ onto  $D_i'$, $\mathcal{A}_i'$, $\mathcal{B}'_i$ and $B_i$ respectively. Moreover, $\psi$ sends $A_0$ onto $A_0'$. The situation is described by the following diagram:

\medskip

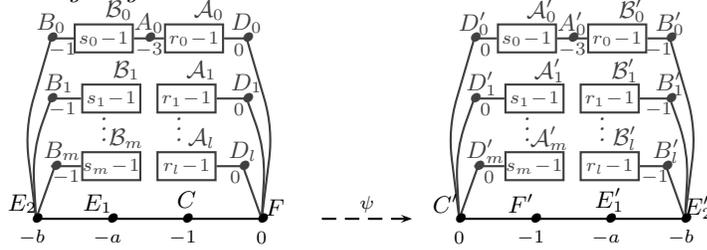
\begin{center}
\begin{pspicture}(0,0.7)(5.5,3.6)
{\darkgray
\psline[linecolor=darkgray](0,1)(0.25,1.7)\rput(0.25,1.7){\textbullet}\rput(0.4,1.55){\scriptsize $-1$}\rput(0.35,1.9){\small $B_m$}
\psline[linecolor=darkgray,linearc=2](0,1)(-.05,1.9)(0.2,2.6)\rput(0.2,2.6){\textbullet}\rput(0.35,2.45){\scriptsize $-1$}\rput(0.3,2.8){\small $B_1$}\psline[linecolor=darkgray,linearc=2](0,1)(-0.15,2.2)(0.2,3.4)
\rput(0.2,3.4){\textbullet}\rput(0.35,3.25){\scriptsize $-1$}\rput(0.2,3.6){\small $B_0$}
\rput(1.5,3.4){\textbullet}\rput(1.5,3.25){\scriptsize $-3$}\rput(1.5,3.6){\small $A_0$}
\psline[linecolor=darkgray](0.2,2.6)(0.6,2.6)
\psframe[linecolor=darkgray](0.6,2.4)(1.4,2.8)
\psline[linecolor=darkgray](0.2,1.7)(0.6,1.7)
\psframe[linecolor=darkgray](0.6,1.5)(1.4,1.9)
\psline[linecolor=darkgray](0.2,3.4)(0.5,3.4)
\psframe[linecolor=darkgray](0.5,3.2)(1.3,3.6)
\psline[linecolor=darkgray](1.3,3.4)(1.7,3.4)
\psframe[linecolor=darkgray](1.7,3.2)(2.5,3.6)
\psline[linecolor=darkgray](2.5,3.4)(2.8,3.4)
\rput(1,1.7){\scriptsize $s_m\!-\!1$}\rput(1.2,2.07){\small $\mathcal{B}_m$}
\rput(0.9,2.3){$\vdots$}
\rput(1,2.6){\scriptsize $s_1\!-\!1$}\rput(1.2,2.97){\small $\mathcal{B}_1$}
\rput(0.9,3.4){\scriptsize $s_0\!-\!1$}\rput(1.1,3.77){\small $\mathcal{B}_0$}
\rput(2.1,3.4){\scriptsize $r_0\!-\!1$}\rput(2.3,3.77){\small $\mathcal{A}_0$}
\psline[linecolor=darkgray](3,1)(2.75,1.7)\rput(2.75,1.7){\textbullet}\rput(2.65,1.55){\scriptsize $0$}\rput(2.75,1.9){\small $D_l$}
\psline[linecolor=darkgray,linearc=2](3,1)(3.05,1.9)(2.8,2.6)\rput(2.8,2.6){\textbullet}\rput(2.7,2.45){\scriptsize $0$}\rput(2.8,2.8){\small $D_1$}
\psline[linecolor=darkgray,linearc=2](3,1)(3.15,2.2)(2.8,3.4)
\rput(2.8,3.4){\textbullet}\rput(2.7,3.25){\scriptsize $0$}\rput(2.8,3.6){\small $D_0$}
\psline[linecolor=darkgray](2.4,2.6)(2.8,2.6)
\psframe[linecolor=darkgray](1.6,2.4)(2.4,2.8)
\psline[linecolor=darkgray](2.4,1.7)(2.8,1.7)
\psframe[linecolor=darkgray](1.6,1.5)(2.4,1.9)
\rput(2,1.7){\scriptsize $r_l\!-\!1$}\rput(2.2,2.07){\small $\mathcal{A}_l$}
\rput(1.9,2.3){$\vdots$}
\rput(2,2.6){\scriptsize $r_1\!-\!1$}\rput(2.2,2.97){\small $\mathcal{A}_1$}
}
\psline(0,1)(3,1)
\rput(0.82,1.2){{\small $E_1$}}
\rput(-0.18,1.2){{\small $E_2$}}
\rput(2,1.25){{\small $C$}}
\rput(3.18,1.15){{\small $F$}}
\rput(0,1){\textbullet}\rput(-0.05,0.75){{\scriptsize $-b$}}
\rput(1,1){\textbullet}\rput(0.95,0.75){{\scriptsize $-a$}}
\rput(2,1){\textbullet}\rput(1.95,0.75){{\scriptsize $-1$}}
\rput(3,1){\textbullet}\rput(3,0.75){{\scriptsize $0$}}
\rput(4.4,1.2){{\scriptsize $\psi$}}
\psline[linestyle=dashed]{->}(3.8,1)(5,1)
\end{pspicture}
\begin{pspicture}(0,0.7)(4,3.6)
{\darkgray
\psline[linecolor=darkgray](0,1)(0.25,1.7)\rput(0.25,1.7){\textbullet}\rput(0.4,1.55){\scriptsize $0$}\rput(0.35,1.9){\small $D_m'$}
\psline[linecolor=darkgray,linearc=2](0,1)(-.05,1.9)(0.2,2.6)\rput(0.2,2.6){\textbullet}\rput(0.35,2.45){\scriptsize $0$}\rput(0.3,2.8){\small $D_1'$}\psline[linecolor=darkgray,linearc=2](0,1)(-0.15,2.2)(0.2,3.4)
\rput(0.2,3.4){\textbullet}\rput(0.35,3.25){\scriptsize $0$}\rput(0.2,3.6){\small $D_0'$}
\rput(1.5,3.4){\textbullet}\rput(1.5,3.25){\scriptsize $-3$}\rput(1.5,3.6){\small $A_0'$}
\psline[linecolor=darkgray](0.2,2.6)(0.6,2.6)
\psframe[linecolor=darkgray](0.6,2.4)(1.4,2.8)
\psline[linecolor=darkgray](0.2,1.7)(0.6,1.7)
\psframe[linecolor=darkgray](0.6,1.5)(1.4,1.9)
\psline[linecolor=darkgray](0.2,3.4)(0.5,3.4)
\psframe[linecolor=darkgray](0.5,3.2)(1.3,3.6)
\psline[linecolor=darkgray](1.3,3.4)(1.7,3.4)
\psframe[linecolor=darkgray](1.7,3.2)(2.5,3.6)
\psline[linecolor=darkgray](2.5,3.4)(2.8,3.4)
\rput(1,1.7){\scriptsize $s_m\!-\!1$}\rput(1.2,2.07){\small $\mathcal{A}_m'$}
\rput(0.9,2.3){$\vdots$}
\rput(1,2.6){\scriptsize $s_1\!-\!1$}\rput(1.2,2.97){\small $\mathcal{A}_1'$}
\rput(0.9,3.4){\scriptsize $s_0\!-\!1$}\rput(1.1,3.77){\small $\mathcal{A}_0'$}
\rput(2.1,3.4){\scriptsize $r_0\!-\!1$}\rput(2.3,3.77){\small $\mathcal{B}_0'$}
\psline[linecolor=darkgray](3,1)(2.75,1.7)\rput(2.75,1.7){\textbullet}\rput(2.6,1.55){\scriptsize $-1$}\rput(2.75,1.9){\small $B_l'$}
\psline[linecolor=darkgray,linearc=2](3,1)(3.05,1.9)(2.8,2.6)\rput(2.8,2.6){\textbullet}\rput(2.65,2.45){\scriptsize $-1$}\rput(2.8,2.8){\small $B_1'$}
\psline[linecolor=darkgray,linearc=2](3,1)(3.15,2.2)(2.8,3.4)
\rput(2.8,3.4){\textbullet}\rput(2.65,3.25){\scriptsize $-1$}\rput(2.8,3.6){\small $B_0'$}
\psline[linecolor=darkgray](2.4,2.6)(2.8,2.6)
\psframe[linecolor=darkgray](1.6,2.4)(2.4,2.8)
\psline[linecolor=darkgray](2.4,1.7)(2.8,1.7)
\psframe[linecolor=darkgray](1.6,1.5)(2.4,1.9)
\rput(2,1.7){\scriptsize $r_l\!-\!1$}\rput(2.2,2.07){\small $\mathcal{B}_l'$}
\rput(1.9,2.3){$\vdots$}
\rput(2,2.6){\scriptsize $r_1\!-\!1$}\rput(2.2,2.97){\small $\mathcal{B}_1'$}
}
\psline(0,1)(3,1)
\rput(0.82,1.2){{\small $F'$}}
\rput(-0.18,1.2){{\small $C'$}}
\rput(2,1.25){{\small $E_1'$}}
\rput(3.18,1.15){{\small $E_2'$}}
\rput(0,1){\textbullet}\rput(0,0.75){{\scriptsize $0$}}
\rput(1,1){\textbullet}\rput(0.95,0.75){{\scriptsize $-1$}}
\rput(2,1){\textbullet}\rput(1.95,0.75){{\scriptsize $-a$}}
\rput(3,1){\textbullet}\rput(2.95,0.75){{\scriptsize $-b$}}
\end{pspicture}
\end{center}
Furthermore, $P'(x)=cQ(dx)$, and $P(x)=eQ'(fx)$ for some $c,d,e,f\in \k^{*}$.
\end{lem}

\begin{proof}
We decompose $\phi$ into $\phi=\theta_2\varphi_2\theta_1\varphi_1\theta_0$ as in $\S\ref{DescriptionReversion2}$, and use this decomposition to see that $E_2$ and $E_2'$ correspond respectively to the curves $\mathcal{E}_p'$ and $\mathcal{E}_p$ obtained by blowing-up $p'$ and $p$. 

According to Lemma~\ref{Lem:P1P2Nota}, there are four possibilities for the situation on $Y$, depending on $P,Q,\lambda$.  We study the image of the curves $D_0,\dots,D_l$, which intersect $F$ at the point $p$. 

Let $i\in \{0,\dots, r\}$ and assume that $D_i\subset Y$ does not intersect the boundary $E_2\cup E_1 \cup C\cup F$ at another point (which occurs in all cases, except for $i=0$ in case $\mathrm{I}a$). In the decomposition of  $\phi$, the curve $D_i$ is affected by the blow-up of $p\in D_i$ and then is not affected by all other maps. In consequence, the image $\phi(D_i)$ of $D_i$ on $Y'$ is a curve that intersects the boundary only at one point, being on $E_2'$, and which has self-intersection $\phi(D_i)^2=(D_i)^2-1$. The curve $\phi(D_i)$ is thus contained in the special fibre and corresponds therefore to one of the $B_i'$ if $(D_i)^2=0$ and to $A_0'$ in case $\mathbf{II}$ if $(D_i)^2=-1$. This shows that we obtain case $\mathbf{II}$ if and only if we start from $\mathbf{I}b$. We can only go to $\mathbf{III}$ if we start from $\mathbf{III}$, because of the special singularity, and then we see that $\mathbf{I}a$ goes to $\mathbf{I}a$.

The diagrams above follow from the discussion made on the image of the $D_i$. It remains to see the correspondence between $P,Q,P',Q',\lambda,\lambda'$. The map $\phi$ induces an isomorphism between the blow-up $\mathcal{E}_p$ of $p$ and the line $E_2'\subset Y'$. This isomorphism sends the tangent direction of $D_i$, which has equation $x-\lambda y=\alpha_i z$, onto $\phi(D_i)\cap E_2$. It also sends the direction of $F$, which is $z=0$, onto $E_2'\cap E_1'$. We obtain therefore an isomorphism $\p^1\to E_2'$ which sends $(0:1)$ onto $E_2'\cap E_1'$ and $(1:\alpha_i)$ onto $E_2'\cap \phi(D_i)$ for each $i$. Studying each of the three diagrams gives $P'$ and $Q'$ in terms of $Q$ and $P$.

In the first diagram (case $\mathbf{I}a$ on both sides), $E_2'$ corresponds to the blow-up of $(0:0:1)$, and the intersection of $B_i'$ with $E_2'$ corresponds to the tangent direction of $x=\beta_i'y$ ($\S\ref{PairConstruction}$). The curve $D_0'$ is the tangent direction of $x=\lambda'y$. We obtain an affine automorphism of $\k$ which sends $\alpha_i$ on $\beta_i'$ for $i=1,\dots,l$ and which sends $0$ onto $\lambda'$.  This means that $P(x)=eQ'(fx+\lambda')$ for some $e,f\in \k^{*}$. Doing the same in the other direction, we obtain $P'(x)=cQ(dx+\lambda)$ for $c,d\in \k^{*}$.

In the second diagram (case $\mathbf{I}b$ on $Y$ and $\mathbf{II}$ on $Y'$), the curve $E_2'$ is the blow-up of the point $(u,v)=(0,0)$ obtained after blowing-up $(0:0:1)$ via $(u,v)\mapsto (u,u^{r_0'}v)$, where $r_0'>0$ is the multiplicity of $0$ in $P'(x)$  (see $\S\ref{PairConstruction}$). The intersection of $B_i'$ with $E_2'$ corresponds to the direction of $u=\beta_i'v$, the point $E_1'\cap E_2'$ corresponds to $v=0$, and $A_0'\cap E_2'$ to $u=0$. We obtain an automorphism of $\k^{*}$ that sends $\alpha_i$ onto $\beta_i'$ for  $i=1,\dots,l$, so $P(x)=eQ'(fx)$ for some $e,f\in \k^{*}$. To obtain $P'(x)$ from $Q(x)$, we do the same computation in the other direction: we have an isomorphism $\p^1\to E_2$ which sends $(0:1)$ onto $E_2\cap E_1$ and $(1:\alpha_i')$ onto $E_2\cap \phi^{-1}(D_i')$ for each $i$. Recall that  $\phi^{-1}(D_0')=B_m$ and $\phi^{-1}(D_i')=B_i$ for $i=1,\dots,m$. The curve $E_2$ corresponds to the blow-up of $(0:0:1)$, the intersection of $B_i$ with $E_2$ corresponds to the direction $x=\beta_i y$, and $B_m=\phi^{-1}(D_0')$ corresponds to $x=\lambda y$, which is $x=\beta_m y$. We get an affine automorphism of $\k$ that sends $\alpha_i'$ onto $\beta_i$ for $i=1,\dots,m-1$ and  $0=\alpha_0'$ onto  $\beta_m=\lambda$.  This implies that $P'(x)=cQ(dx+\lambda)$ for some $c,d\in \k^{*}$.

The last diagram is when $Y,Y'$ are in case $\mathbf{III}$, and is symmetrical. As in the second diagram, the curve $E_2'$ is the blow-up of the point $(u,v)=(0,0)$ obtained after blowing-up $(0:0:1)$ via $(u,v)\mapsto (u,u^{r_0'}v)$, where $r_0'>0$ is the multiplicity of $0$ in $P'(x)$. The intersection of $B_i'$ with $E_2'$ corresponds to the direction of $u=\beta_i'v$, the point $E_1'\cap E_2'$ corresponds to $v=0$, and $A_0'\cap E_2'$ to $u=0$. We obtain an automorphism of $\k^{*}$ that sends $\alpha_i$ onto $\beta_i'$ for  $i=1,\dots,l$, so $P(x)=eQ'(fx)$ for some $e,f\in \k^{*}$. And in the other direction, we get $P'(x)=cQ(dx)$ for some $c,d\in \k^{*}$.
\end{proof}

To conclude this section, we give a complete characterization of when two reversions are equivalent in the sense of Definition~\ref{A1fibGraph}, which will be needed in the next section to describe the graphs associated to the surfaces $X\setminus B$.

\begin{prop}\label{Prop:EquivReversions}
Let $(X,B=F\tr C\tr E_1\tr E_2)$ be a pair constructed from polynomials $P,Q\in \k[w]$. For two reversions $\phi_i\colon (X,B)\dasharrow (X_i,B_i)$, $i=1,2$, the following are equivalent:
\begin{enumerate}
\item
The pairs $(X_1,B_1)$ and $(X_2,B_2)$ are isomorphic.
\item
The reversions $\phi_1,\phi_2$ are equivalent, i.e.\ there exists  $\theta\in \Aut(X,B)$ and an isomorphism $\theta':(X_1,B_1)\rightarrow (X_2,B_2)$, such that $\phi_2\circ\theta=\theta'\circ\phi_1$.
\end{enumerate}
Moreover, these equivalent properties are always satisfied if $P(0)=0$.
\end{prop}
\begin{proof}
The implication $(2)\Rightarrow (1)$ is obvious. Conversely, we may suppose that $\phi_1,\phi_2$ are respectively centred at points  $p_1,p_2\in F\setminus C$ which we identify in turn with the points $(\lambda_1:1:0),(\lambda_2:1:0)\in \p^2$ (see \S \ref{PairConstruction}). We denote by $(P_1,Q_1)$ and $(P_2,Q_2)$ the polynomials associated to the pairs $(X_1,B_1)$ and $(X_2,B_2)$. Since the reversions are  uniquely determined by the choice of their proper base-point, assertion $(2)$ is equivalent to the existence of an automorphism $\theta\in \Aut(X,B)$ which sends $p_1$ onto $p_2$. If $P(0)=0$, then the automorphism $(x,y)\mapsto (x+(\lambda_2-\lambda_1)y, y)$ of $\A^2$ lifts to an automorphism $\theta\in \Aut(X,B)$ (Proposition~\ref{Prop:IsoPairs}) such that $\theta(p_1)=p_2$.
So it remains to consider the case where $P(0)\not=0$ (case $\mathbf{I}$). By Lemma~\ref{Lem:Rev01k} (assertions $(1)$ and $(2)$), we have $P_i(w)=c_iQ(d_i w+\lambda_i)$, for some $c_i,d_i\in \k^{*}$. Since $(X_1,B_1)$ and $(X_2,B_2)$ are isomorphic we also have $P_1(w)=\alpha P_2(\beta w)$, for some $\alpha,\beta\in \k^{*}$ (Proposition~\ref{Prop:IsoPairs}). This yields 
$$c_1Q(d_1 w+\lambda_1)=P_1(w)=\alpha P_2(\beta w)=\alpha c_2Q(d_2 \beta w+\lambda_2),$$
which implies (replacing $w$ by $(w-\lambda_1)/d_1$)  that $Q(d_2\beta (w-\lambda_1)/d_1+\lambda_2)/Q(w)=\alpha c_2/c_1\in \k^{*}$. Letting $c=\frac{d_1}{d_2\beta}$ and $b=\lambda_1-\lambda_2c$, we obtain that $Q(\frac{w-b}{c})/Q(w)\in \k^{*}$ and $c\lambda_2=\lambda_1+b$. The first condition guarantees that the automorphism $\nu\colon (x,y)\mapsto (x+by,cy)$ of $\A^2$ lifts to an automorphism of $(X,B)$ (see Proposition~\ref{Prop:IsoPairs}) while the second equality says precisely that the extension of $\nu$ to $\p^2$ maps $p_1=[\lambda_1:1:0]$ onto $p_2=[\lambda_2:1:0]$. 
\end{proof}
\begin{rem}  Proposition~\ref{Prop:EquivReversions} implies in particular that in the graph $\mathcal{F}_S$ associated to $S=X\setminus B$ as in Definition~\ref{A1fibGraph}, two arrows corresponding to reversions are equal if and only if they have the same source and target. Consequently, the graph $\mathcal{F}_S$ does not contain any cycle of length $2$ and each of its vertices is the base vertex of at most one cycle of length $1$.
\end{rem}

\begin{example} \label{ChangingModels} As explained in $\S~\ref{SurfEquations}$, given polynomials $P,Q\in\k[w]$ of degrees $\geq 1$, the surface $S$ in $\mathbb{A}^4=\mathrm{Spec}(\k[x,y,u,v])$ defined by the system of equations
$$\left\{\begin{array}{lcl}
yu&=&xP(x)\\
vx&=&uQ(u)\\
yv&=&P(x)Q(u)\end{array}\right.$$
comes equipped with the $\mathbb{A}^1$-fibration $\pi=\mathrm{pr}_y\mid_S$ induced by the restriction of the rational pencil $\bar{\pi}$ on the standard pair $(X,B=F\tr C\tr E,\bar{\pi})$ associated with the pair $(P,Q)$ via the construction of $\S~\ref{PairConstruction}$. 

The automorphism $(x,y,u,v)\mapsto (u,v,x,y)$ of $\A^4$ induces an isomorphism $\sigma$ of $S$ with the surface $S'\subset \A^4$ defined by the system of equations
$$\left\{\begin{array}{lcl}
yu&=&xQ(x)\\
vx&=&uP(u)\\
yv&=&Q(x)P(u),\end{array}\right.$$
which comes equipped  with the $\mathbb{A}^1$-fibration $\pi'=\mathrm{pr}_y\mid_{S'}$ induced by the restriction of the rational pencil $\bar{\pi}'$ on the standard pair $(X',B'=F'\tr C'\tr E',\bar{\pi}')$ associated with the pair $(P',Q')=(Q,P)$ via the construction of $\S~\ref{PairConstruction}$.

With our choice of coordinates, the closure in $X$ of the general fibers of $\pi'\circ \sigma=\pr_v \mid_S$ all intersect $B$ at the point $p\in F$ with image $\tau(p)=[0:1:0]\in \mathbb{P}^2$, and one checks that the birational map of standard pairs $(X,B)\dashrightarrow (X',B')$ corresponding to $\sigma$ is a reversion centered at $p$.

Note that if $P(0)\not=0$ and $Q(0)=0$ then $S$ equipped with $\pi$ is of type {\bf I} while it is of type {\bf II} when equipped with $\sigma\pi'=\pr_v \mid_S$. 
\end{example}

\section{Graphs of $\mathbb{A}^1$-fibrations and associated graphs of groups}
   Here we apply the results of the previous section to characterize equivalence classes of $\mathbb{A}^1$-fibrations on normal affine surfaces admitting a completion by a standard pair $(X,B)$ of type $(0,-1,-a,-b)$. We also give explicit description of automorphism groups of certain of these surfaces.

\begin{enavant} {\bf Notation.} Given polynomials $P,Q\in\k[w]$, we denote by $[P,Q]$ the isomorphism class of the standard pair $(X,B,\bar{\pi})$ obtained obtained by means of the construction of $\S~\ref{PairConstruction}$. By virtue of Proposition~\ref{Prop:IsoPairs}, $[P,Q]=[P',Q']$ if and only if the corresponding $\mathbb{A}^1$-fibered surfaces $(X\setminus B,\bar{\pi}\mid_{X\setminus B})$ and $(X'\setminus B',\bar{\pi}'\mid_{X'\setminus B'})$ are isomorphic. Recall that by virtue of Proposition~\ref{Prop:IsoPairs}, this holds if and only if  $P'(w)=\alpha P(\beta w)$, $Q'(w)=\gamma Q(\delta w+t)$ , where $\alpha,\beta,\gamma,\delta\in \k^{*}$ and $t\in k$ being $0$ if $P(0)=0$.

 We say that $[P,Q]$ is \emph{equivalent} to $[P',Q']$ if $X\setminus B$ and $X'\setminus B'$ are isomorphic as abstract affine surfaces. With this convention, the vertices of the graph of $\mathbb{A}^1$-fibrations $\mathcal{F}_{X\setminus B}$ of $X\setminus B$ as defined in $\S\ref{A1fibGraph}$ are in one-to-one correspondence with pairs $[P',Q']$ equivalent to $[P,Q]$. In what follows we denote this graph simply by $\mathcal{F}_{[P,Q]}$.  
\end{enavant}

 Note that arrows of the graph $\mathcal{F}_{[P,Q]}$ correspond to equivalence classes of reversions between pairs equivalent to $[P,Q]$. Given one such pair $[P',Q']$, represented by a pair $(X',B'=F'\tr C'\tr E_1'\tr E_2',\overline{\pi}')$, the possible reversions starting from it are parametrized by the $\k$-rational points of the line $F'\setminus C'$. Moreover, if $\sigma_1,\sigma_2$ are two reversions centred at points $p_1,p_2\in F'\setminus C'$, they are equivalent, or give the same arrow, (see Definition~\ref{A1fibGraph}) if and only if there exists an automorphism of $(X',B')$ that sends $p_1$ onto $p_2$. By Proposition~\ref{Prop:IsoPairs}, this always holds when $P(0)=0$. 

\subsection{Affine surfaces of type III}
\indent\newline\noindent As noted above normal affine surfaces corresponding to case {\bf III} in the the construction of $\S~\ref{PairConstruction}$ are always singular and form a distinguished class stable under taking reversions. For such a surface $S$ in $\mathbb{A}^4=\mathrm{Spec}(\k[x,y,u,v])$ defined by a system of equations
$$\left\{\begin{array}{lcl}
yu&=&xP(x)\\
vx&=&uQ(u)\\
yv&=&P(x)Q(u)\end{array}\right.$$
corresponding to a pair $[P,Q]$ with $P(0)=Q(0)=0$, the structure of the graph $\mathcal{F}_{[P,Q]}$ is particularly simple: Indeed, Lemma~\ref{Lem:Rev01k} implies that $[Q,P]$ is the only pair equivalent to $[P,Q]$ and that they can be obtained from each other by performing a reversion. Since $[P,Q]=[Q,P]$ if and only if $Q=\alpha P(\beta w)$ for some $\alpha,\beta\in\k^*$, the corresponding $\mathcal{F}_{[P,Q]}$ is thus 

 $$\xymatrix@R=1mm@C=2cm{
\ar@(ul,dl){ } 
}[P,Q] \;\textrm{ if } Q(w)=\alpha P(\beta w),\; \alpha,\beta\in \k^{*} \quad \textrm{ or } \quad \xymatrix@R=1mm@C=2cm{
[P,Q]\ar@{<->}[r]& [Q,P]
} \textrm{ otherwise }$$

\indent\newline\noindent Denoting by $J_y=\Aut(S,\pr_y)$ and $J_v=\Aut(S,\pr_v)$ the groups of automorphisms of $S$ which preserve the $\mathbb{A}^1$-fibrations $\pr_y\colon S\to \A^1$ and $\pr_v\colon S\to \A^1$ respectively and by $\mathrm{Diag}(S)\subset \Aut(S)$  the subgroup consisting of restrictions to $S$ of diagonal automorphisms of $\A^4$ preserving $S$, we obtain the following description of automorphism groups of affine surfaces of type {\bf III}:

\begin{prop}
For an affine surface $S$ in $\mathbb{A}^4=\mathrm{Spec}(\k[x,y,u,v])$ defined by the equations
$$\left\{\begin{array}{lcl}
yu&=&xP(x)\\
vx&=&uQ(u)\\
yv&=&P(x)Q(u).\end{array}\right .$$
where $P,Q$ are non constant polynomials with $P(0)=Q(0)=0$, the following holds: 
\begin{enumerate}
\item 
Every $\mathbb{A}^1$-fibration on $S$ is conjugated either to $\pr_y\colon S\to \A^1$ or $\pr_v\colon S\to \A^1$ and these two fibrations are isomorphic to each other if and only if $Q(w)=\alpha P(\beta w)$ for some $\alpha,\beta\in \k^{*}$.
\item
If $Q(w)=\alpha P(\beta w)$ for some $\alpha,\beta\in \k^{*}$, then $[P,Q]=[P,P]$ and, assuming further that $Q=P$, the group $\Aut(S)$ is the amalgamated product $A\star_{\mathrm{Diag}} J_y$ of $J_y=\Aut(S,\pr_y)$ and the subgroup $A$ of $\Aut(S)$ generated by $\mathrm{Diag}(S)$ and the involution $\sigma\colon (x,y,u,v)\to (u,v,x,y)$. 
 \item
If $Q(w)\not=\alpha P(\beta w)$, for any $\alpha,\beta\in \k^{*}$, then $J_y\cap J_v=\mathrm{Diag}(S)$ and $\Aut(S)$ is the amalgamated product $J_y\star_{\mathrm{Diag}(S)} J_v$ of $J_y=\Aut(S,\pr_y)$ and $J_v=\Aut(S,\pr_v)$.
\end{enumerate}
\end{prop}
\begin{proof}
The first assertion is an immediate consequence of the description of $\mathcal{F}_S=\mathcal{F}_{[P,Q]}$. Similarly as in  Example~\ref{ChangingModels}, we consider $S$ as $X\setminus B$ where $(X,B=F\tr C\tr E,\bar{\pi})$ is associated with the pair $(P,Q)$ via the construction of $\S~\ref{PairConstruction}$ in such a way that $\pr_y\mid_S$ coincides with $\bar{\pi}\mid_S$. We denote by $\sigma\colon (X,B)\dasharrow (X',B')$ the reversion corresponding to the morphism $(x,y,u,v)\mapsto (u,v,x,y)$ where $(X',B')$ is the standard pair associated with the pair of polynomials $(Q,P)$. By virtue of Proposition~\ref{Prop:IsoPairs}, elements of $\Aut(X,B)$ are lifts of automorphisms of $\mathbb{A}^2$ the form $(x,y)\mapsto (ax+by,cy)$ satisfying $P(aw)/P(w)\in \k^{*}, Q(\frac{a^{r_0+1}}{c}\cdot w)/Q(w)\in \k^{*}$, where $r_0\ge 1$ is the multiplicity of $0$ as a root of $P$. The extension of such an automorphisms to $\p^2$ fixes the center $[1:0:0]$ of $\sigma$ if and only if $b=0$. If so, we write $\lambda=P(aw)/P(w)=a^{r_0}$ and $\mu=Q(\frac{a^{r_0+1}}{c}\cdot w)/Q(w)=Q(\frac{a\lambda}{c}w)/Q(w)$, and check that the lift to $S$ of the corresponding automorphism coincides with the restriction of the diagonal automorphism $(x,y,u,v)\mapsto (ax,cy,\frac{a\lambda}{c} u, \frac{\lambda\mu}{c}v)$ of $\A^4$. Furthermore, every diagonal automorphism of $\A^4$ which preserves $S$ is necessarily of this form. This implies in particular the group $\mathrm{Diag}(S)$ coincides precisely with the subgroup of $\Aut(X,B)$ consisting of lifts of automorphisms whose extensions to $\p^2$ fix the point $[1:0:0]$. 
By Proposition~\ref{mainfirst}, every birational map $f\colon (X,B)\dasharrow (X,B)$ is either an element of $\Aut(X,B)$ (and in this case belongs to $J_y$) or decomposes into a finite sequence of fibered modifications and reversions. Since all reversions are equivalent to $\sigma$ or $\sigma^{-1}$, we can assume that $f$ is a product of $\sigma$, $\sigma^{-1}$, and  fibered modifications and automorphisms of the pairs $(X,B)$ and $(X',B')$.

$(2)$ If $Q(w)=\alpha P(\beta w)$ for some $\alpha,\beta\in \k^{*}$, we can assume further that $Q=P$ so that $\sigma$ becomes in fact an automorphism of $S$. This implies that $\Aut(S)$ is generated by $J_y$ and $\sigma$, and hence by $J_y$ and $A=\langle \mathrm{Diag}(S),\sigma\rangle$. It remains to see that every element $h=j_ma_m\dots j_2a_ 2 j_1a_1$ with $a_l\in A\setminus J_y$, $j_l\in J_y\setminus A$ is not trivial. By definition every $j_l\in J_y\setminus A$ is either a fibered modification $(X,B)\dasharrow (X,B)$ or an automorphism which does not fix the center of the reversion $\sigma$. On the other hand, every $a_l\in A\setminus J_y$ is a reversion $(X,B)\dasharrow (X,B)$ which has the same center as $\sigma$. It follows that $h$ is either an element of $A\setminus J_y$ or admits
a reduced decomposition containing at least a reversion. So $h$ is never trivial, as desired.

$(3)$ If $Q(w)\not=\alpha P(\beta w)$ for any $\alpha,\beta\in \k^{*}$, then $(X',B')$ is not isomorphic to $(X,B)$. In particular, every element $f\colon (X,B)\dasharrow (X',B')$ decomposes into  
$$f=\sigma^{-1} a_n'\sigma \dots \sigma^{-1}a_2'\sigma_2a_2\sigma^{-1} a_1'\sigma a_1,$$
where the $a_i$ and $a_j'$ are either fibered modifications or automorphisms of $(X,B)$ and $(X',B')$ respectively. In consequence, every $a_i$ is an element of $J_y$ and every $\sigma^{-1}a_i'\sigma$ is an element of $J_v$, which shows that $\Aut(S)$ is generated by $J_y$ and $J_v$. Furthermore, $a\in J_y$ is an element of $J_v$ if and only if it is equal to $\sigma^{-1} b\sigma$ for a certain automorphism or a fibered modification $b$ of $(X',B')$. The equality $a=\sigma^{-1} b\sigma$ implies that $a$ and $b$ are automorphisms of $(X,B)$ and $(X',B')$ respectively, and that $a$ (respectively $b$) fixes the center of $\sigma$ (respectively of $\sigma^{-1}$). In particular, $J_y \cap J_v=\mathrm{Diag}(S)$.

It remains to show that every element $h=b_ma_m\dots b_2a_ 2 b_1a_1$ with $a_l\in J_y\setminus J_v$, $b_l=\sigma^{-1} b_l' \sigma\in J_v\setminus J_y$ is not trivial. By virtue of the above description, each $a_i$ is either an automorphism of $(X,B)$ not fixing the center of $\sigma$ or a fibered modification while each $b_j'$ is either an automorphism of $(X',B')$ not fixing the center of $\sigma^{-1}$ or a fibered modification of $(X',B')$. This implies that $h$ is either an element of $\Aut(X,B)$ not fixing the center of $\sigma$ or admits a reduced decomposition containing at least a reversion. In any case, $h$ is not trivial which achieves the proof. 
\end{proof}

\begin{example} Let $S$ be the surface in $\mathbb{A}^4=\mathrm{Spec}(\k[x,y,u,v])$ defined by the equations
$$\left\{ \begin{array}{lcl}
yu&=&x^2(x-1)\\
vx&=&u^2(u-1)\\
yv&=&x(x-1)u(u-1)\end{array} \right.$$
corresponding to the polynomials $P(w)=Q(w)=w(w-1)$. Since $P(aw)/P(w)$ and $Q(\frac{a^2}{c}\cdot w)/Q(w)$ belong to $\k^*$ if and only if $a=c=1$, it follows from the proof of the above proposition that $\mathrm{Diag}(S)=\{\mathrm{id}_S\}$. Furthermore, one checks that the group $J_y=\Aut(S,\pr_y)$ consist of lifts to $S$ via the projection $\mathrm{pr}_{x,y}:S\rightarrow \mathbb{A}^2=\mathrm{Spec}(\k[x,y])$ of automorphisms of $\mathbb{A}^2$ of the form $(x,y)\mapsto (x+y^2R(y),y)$, where $R\in \k[y]$ is an arbitrary polynomial. So $J_y$ is isomorphic as a group to $(\k[y],+)\simeq \mathbb{G}_{a,k}^{\infty}$ and we conclude that $\Aut(S)$ is isomorphic to the free product $\mathbb{Z}_2\star \mathbb{G}_{a,k}^{\infty}$. 
\end{example}

\begin{example} Let $S$ be the surface in $\mathbb{A}^4=\mathrm{Spec}(\k[x,y,u,v])$ defined by the equations
$$\left\{ \begin{array}{lcl}
yu&=&x^2(x-1)\\
vx&=&u^(u-1)\\
yv&=&x(x-1)(u-1)\end{array} \right.$$
corresponding to the polynomials $P(w)=w(w-1)$ and $Q(w)=w-1$. Again, the choice of $P$ and $Q$ guarantees that $\mathrm{Diag}(S)=\{\mathrm{id}_S\}$. Similarly as in the previous example, the groups $J_y=\Aut(S,\pr_y)$ and $J_v=\Aut(S,\pr_v)$ consists of lifts to $S$ via the projections $\mathrm{pr}_{x,y}$ and $\mathrm{pr}_{u,v}$ respectively of automorphisms of $\mathbb{A}^2$ of the form $(x,y)\mapsto (x+y^2R_1(y),y)$, where $R_1\in \k[y]$, and $(u,v)\mapsto (u+vR_2(v),v)$ where $R_2\in \k[v]$. It follows that $\Aut(S)$ is isomorphic to the free product $\mathbb{G}_{a,k}^{\infty} \star \mathbb{G}_{a,k}^{\infty}$ of two copies of $\mathbb{G}_{a,k}^{\infty}$. 
\end{example}

\subsection{Affine surfaces of types I and II}
\indent\newline\noindent In contrast with surfaces of type {\bf III}, Example \ref{ChangingModels} shows that in general affine surfaces corresponding to case {\bf I} and {\bf II} in the construction of $\S~\ref{PairConstruction}$ can be obtained from each other by performing reversions. Recall that these models correspond to pairs $[P,Q]$  such that either $P(0)\neq 0$ or $P(0)=0$ but $Q(0)\neq 0$. The associated graph $\mathcal{F}_{[P,Q]}$ is quite complicated, in particular infinite as soon as the field $\k$ is, as shown by the following result:

\begin{prop}\label{Prop:Tapis}
Let $P,Q$ be two polynomials of degree $\ge 1$, and assume that $P(0)\not=0$. The set of pairs equivalent to $[P,Q]$ is $$\left\{[P(w+a),Q(w+b)],[Q(w+b),P(w+a)]\ \Big|\ a,b\in \k, (P(a),Q(b))\not=(0,0)\right\}$$
The graph  $\mathcal{F}_{[P,Q]}$ associated to $[P,Q]$ has the following structure
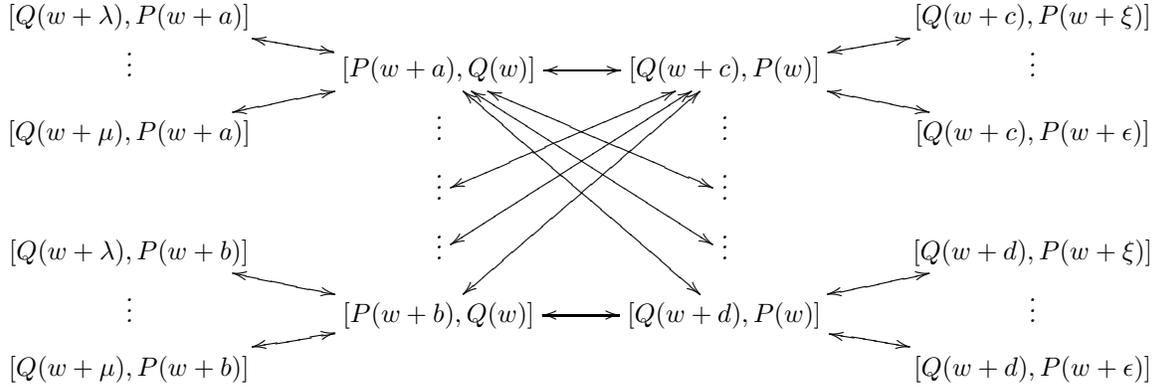
\[\xymatrix@R=0.01cm@C=1cm{
[Q(w+\lambda),P(w+a)]\ar@{<->}[rd]&&&[Q(w+c),P(w+\xi)]\\
\vdots& [P(w+a),Q(w)]\ar@{<->}[r]\ar@{<->}[rdddd]\ar@{<->}[rddd]\ar@{<->}[rdd]& \ar@{<->}[ldddd]\ar@{<->}[lddd]\ar@{<->}[ldd][Q(w+c),P(w)]\ar@{<->}[rd]\ar@{<->}[ru]&\vdots\\
[Q(w+\mu),P(w+a)]\ar@{<->}[ru]&\vdots&\vdots&[Q(w+c),P(w+\epsilon)]\\
&\vdots&\vdots\\
[Q(w+\lambda),P(w+b)]\ar@{<->}[rd]&\vdots&\vdots&[Q(w+d),P(w+\xi)]\\
\vdots& [P(w+b),Q(w)]\ar@{<->}[r]& [Q(w+d),P(w)]\ar@{<->}[rd]\ar@{<->}[ru]&\vdots\\
[Q(w+\mu),P(w+b)]\ar@{<->}[ru]&&&[Q(w+d),P(w+\epsilon)]}\]
where $P(a)P(b)Q(c)Q(d)\not=0$, $Q(\lambda)=Q(\mu)=P(\xi)=P(\epsilon)=0$.
\end{prop}
\begin{proof}
Follows from Lemma~\ref{Lem:Rev01k} and Proposition~\ref{Prop:EquivReversions}: 
\begin{enumerate}
\item
Starting from $[P(w+a),Q(w+b)]$ with $P(a)\not=0$ and performing any reversion, we obtain all $[Q(w+c),P(w+a)]$ for  $c\in \k$.
\item
Starting from $[P(w+a),Q(w+b)]$ with $P(a)=0$ and $Q(b)\not=0$ and performing any reversion, we only obtain  $[Q(w+b),P(w+a)]$, which is equivalent to $[Q(w+b),P(w+c)]$ for any  $c\in \k$.
\end{enumerate}
This yields the result.
\end{proof}
\begin{prop} Let $\k$ be an uncountable field and let $P,Q\in\k[w]$ be polynomials having at least $2$ distinct roots in an algebraic closure $\overline{\k}$ of $\k$ and such that $P(0)\not=0$. Then for the affine surface $S$ in $\mathbb{A}^4=\mathrm{Spec}(\k[x,y,u,v])$ defined by the system of equations
$$\left\{\begin{array}{lcl}
yu&=&xP(x)\\
vx&=&uQ(u)\\
yv&=&P(x)Q(u),\end{array} \right.$$
the following holds:
\begin{enumerate}
\item $S$ admits uncountably many isomorphism classes of $\A^1$-fibrations.

\item The subgroup $H\subset \Aut(S)$ generated by all automorphisms of $\A^1$-fibrations is not generated by a countable union of algebraic groups.

\item The subgroup $\Aut(S)_{\mathrm{alg}}\subset \Aut(S)$ generated by all algebraic subgroups of $\Aut(S)$ is is not generated by a countable union of algebraic groups.

\item The quotient of $\Aut(S)$ by its normal subgroup $\Aut(S)_{\mathrm{alg}}$ contains a free group over an uncountable set of generators. Furthermore, the same holds for $\Aut(S)/H\simeq \Pi_1(\mathcal{F}_S)$ since $H\subset \Aut(S)_{\mathrm{alg}}$. 
\end{enumerate}
\end{prop}

\begin{proof}
The conditions on $P$ and $Q$ imply that $P(\alpha w+\beta)/P(w)\in \k^{*}$ for only finitely many pairs $(\alpha,\beta)\in \k^{*}\times \k$, and the same holds for $Q$. Furthermore, the fact that the degree of $P$ and $Q$ is at least $2$ implies that for every standard pair $(X,B)$ with $X\setminus B\simeq S$, the boundary $B$ contains at least an irreducible component with self-intersection $\le -3$.

Suppose that $Q(w)= \alpha P(\beta w)$ for some $\alpha,\beta\in \k^{*}$.  Choosing $t\in \k$ general enough, we then have $P(\gamma(w+t))/P(w)\not\in \k^*$ for any $\gamma\in \k^{*}$. Applying Proposition~\ref{Prop:IsoPairs}, we can replace $Q(w)$ with $P(w+t)$, and obtain that $Q(w)\not= \alpha P(\beta w)$ for any $\alpha,\beta\in \k^{*}$, which implies that $[(P(w),Q(w))]\not=[(Q(w),P(w))]$.

We can now choose an uncountable set $A\subset \k$, containing $0$, such that for every distinct $a_1,a_2\in A$, 
we have $[P(w+a_1),Q(w)]\not=[P(w+a_2),Q(w)]$ and  $[Q(w+a_1),P(w)]\not=[Q(w+a_2),P(w)]$ and such that for any $a\in A$, the four equivalence classes $[P(w),Q(w)]$, $[Q(w),P(w)]$, $[P(w+a),Q(w)]$ and $[Q(w+a),P(w)]$ are distinct.

For every $a\in A$, we denote by $(X_a,B_a)$ the standard pair obtained from the pair of polynomials $(P(w+a),Q(w))$ and  by $(X_a',B_a')$ the one obtained from the pair of polynomials $(Q(w+a),P(w))$. By Proposition~\ref{Prop:Tapis}, each $[(X_a,B_a)]$ and each $[(X_a',B_a')]$ is a vertex in the graph $\mathcal{F}_S=\mathcal{F}_{[P,Q]}$. The choice made on $A$ implies that the four pairs $[(X_a,B_a)]$, $[(X_b,B_b)]$, $[(X_a',B_a')]$ and $[(X_b',B_b')]$ are pairwise distinct for distinct $a,b\in A$. In particular, we obtain uncountably many vertices, which is is equivalent to $(1)$ by Proposition~\ref{Prop:IsoPairs}. 

Let $(G_i)_{i\in \mathbb{N}}$ be a countable set of algebraic subgroups $G_i\subset \Aut(S)$. For any $i\in \mathbb{N}$,  Proposition~\ref{Prop:AlgGroups}, gives a  standard pair $(X,B)$ (depending on $i$), an an isomorphism $\psi:S\stackrel{\sim}{\rightarrow} X\setminus B$ such that the conjugation of $G_i$ by $\psi$ consists of birational maps $(X,B)\dasharrow (X,B)$ being either fibered modifications, automorphisms of self-intersections. Viewing any element of $G_i$ as a birational map $(X_0,B_0)$, we can decompose it into automorphisms of pairs, fibered modifications and reversions (Proposition~\ref{mainfirst}), and the existence of $\psi$ implies that the number of such maps is then bounded. In consequence, there exists a countable set $\mathcal{S}$ of equivalence classes of pairs $(X,B)$ with $X\setminus B=S$, such that each element of each $G_i$ can be decomposed into a sequence of automorphisms of pairs, fibered modifications and reversions involving only pairs in $\mathcal{S}$. There exists thus $a\in A$ such that $[(X_a',B_a')]\notin \mathcal{S}$. We choose a reversion $\mu\colon (X_0,B_0)\dasharrow (X_a',B_a')$, and an algebraic group $\hat{G}$ of fibered modifications $(X_a',B_a')\dasharrow (X_a',B_a')$. The group $\mu^{-1}\hat{G}\mu$ yields an algebraic subgroup of $\Aut(S)$, which preserves an $\A^1$-fibration, and which is not contained in the group generated by the $G_i$. This yields $(2)$ and $(3)$.

For every $a\in A$, there exist reversions $\tau\colon (X_0,B_0)\dasharrow (X_0',B_0')$, $\sigma_{a}\colon  (X_0',B_0')\dasharrow (X_a,B_a)$, $\tau_a\colon (X_a,B_a)\dasharrow (X_a',B_a')$ and $\sigma'_a\colon (X_a',B_a')\dasharrow (X_0,B_0)$, representing the cycle 

\[\xymatrix@R=1cm@C=1cm{
[P(w),Q(w)]\ar[r]&[Q(w),P(w)]\ar[d]\\
[Q(w+a),P(w)]\ar[u]&[P(w+a),Q(w)]\ar[l]}\]
in $\mathcal{F}_S=\mathcal{F}_{P,Q}$. For every $a\in A\backslash \{0\}$, the map $\sigma'_a\tau_a\sigma_a\tau\colon (X_0,B_0)\dasharrow (X_0,B_0)$ restricts to an automorphism $\zeta_a$ of $S=X_0\setminus B_0$. The decomposition $\sigma'_a\tau_a\sigma_a\tau$ is reduced, because $[P(w),Q(w)]\not= [P(w+a),Q(w)]$ and $[Q(w),P(w)]\not=[Q(w+a),P(w)]$  (recall that the composition of two reversions is of length $2$ or is an isomorphism of pairs). We denote by $F\subset \Aut(S)$ the group generated by the $\zeta_a, a\in A\backslash \{0\}$, and will show that this one is the free group over the $\zeta_a$ and intersects $\Aut(S)_{\mathrm{alg}}$ trivially, in order to get $(4)$.
\begin{enumerate}
\item
First we observe that for every $a,b\in A\backslash \{0\}$, $a\not=b$, the decomposition
$$\zeta_a\zeta_b=\sigma'_a\tau_a\sigma_a\tau\sigma'_b\tau_b\sigma_b\tau$$
is reduced. Indeed, $\tau\sigma'_b$ is reduced because otherwise it would be an isomorphism between $(X'_b,B'_b)$ and $(X_0',B_0')$, which is not possible since $b\not=0$.
\item
Similarly, the following decomposition
$$\zeta_a(\zeta_b)^{-1}=\sigma'_a\tau_a\sigma_a(\sigma_b)^{-1}(\tau_b)^{-1}(\sigma'_b)^{-1}$$
is reduced, because otherwise $\sigma_a(\sigma_b)^{-1}$ would be an isomorphism betwen $(X_b,B_b)$ and $(X_a,B_a)$.
\item
The last case gives the following decomposition
$(\zeta_a)^{-1} \zeta_b=\tau^{-1}(\sigma_a)^{-1} (\tau_a)^{-1} (\sigma'_a)^{-1}\sigma'_b\tau_b\sigma_b\tau,$
which is again reduced, for otherwise $(\sigma'_a)^{-1}\sigma'_b$ would be an isomorphism between $(X'_b,B'_b)$ and $(X'_a,B'_a)$.
\end{enumerate}

These three observations implies that every element $h=(\zeta_{a_r})^{\delta_r}\cdots (\zeta_{a_1})^{\delta_1}\in F$, where $a_1,\dots,a_r\in A$ and $\delta_1,\dots,\delta_r\in \mathbb{Z}\backslash \{0\}$ and $a_i\not=a_{i+1}$ for $i=1,\dots,r-1$, is not trivial since it admits a reduced decomposition of positive length. This shows the freeness of $F$. By construction,  the image of $h$ in $\Pi_1(\mathcal{F}_S)$ consists of a product of loops based at $[(X_0,B_0)]$ of length $\ge 4$. Since in contrast, the image in $\Pi_1(\mathcal{F}_S)$ of every element in $\Aut(S)_{\mathrm{alg}}$ can only contain loops of length $1$ (see Remark~\ref{RemaLoop1}), it follows that $F\cap \Aut(S)_{\mathrm{alg}}$ is trivial, which completes the proof.
\end{proof}

\bibliographystyle{amsplain}

\providecommand{\bysame}{\leavevmode\hbox to3em{\hrulefill}\thinspace}
\providecommand{\MR}{\relax\ifhmode\unskip\space\fi MR }
\providecommand{\MRhref}[2]{%
  \href{http://www.ams.org/mathscinet-getitem?mr=#1}{#2}
}
\providecommand{\href}[2]{#2}
\begin{thebibliography}{}

\end{thebibliography}


\begin{thebibliography}{blabla}
\bibitem{first}
J. Blanc, A. Dubouloz, \emph{Automorphisms of $\mathbb{A}^1$-fibered affine surfaces}.
Trans. Amer. Math. Soc. {\bf 363} (2011), 5887-5924. 

\bibitem{Gi-Da1}
V. I. Danilov, M. H. Gizatullin, {\it  Automorphisms of affine surfaces. I.}  Izv. Akad. Nauk SSSR Ser. Mat.  {\bf 39}  (1975), no. 3, 523-565.

\bibitem{Gi-Da2}
 V. I. Danilov, M. H. Gizatullin, {\it  Automorphisms of affine surfaces. II.} Izv. Akad. Nauk SSSR Ser. Mat. {\bf 41} (1977), no. 1, 54-103.


\bibitem{Gi} 
M. H. Gizatullin,  {\it Quasihomogeneous affine surfaces},  Izv. Akad. Nauk SSSR Ser. Mat. {\bf  35}  (1971), 1047-1071.


\bibitem{Jun}
H.W.E. Jung, \emph{\"Uber ganze birationale Transformationen der Ebene}. J. reine angew. Math. {\bf 184} (1942), 161-174.

\bibitem{ML90}
L. Makar-Limanov, {\it On groups of automorphisms of a class of surfaces}, Israel J. Math. {\bf 69} (1990), no. 2, 250-256.

\bibitem{MaOo67} 
H. Matsumura, F. Oort, \emph{Representability of group functors, and automorphisms of algebraic schemes}. Invent. Math. 4 (1967) 1-25.

\bibitem{MiyBook}
M. Miyanishi, {\it Open {A}lgebraic {S}urfaces}, CRM Monogr. Ser., 12, Amer. Math. Soc., Providence, RI, 2001.

\bibitem{Ram64} 
Ramanujam, C.P., \emph{A Note on Automorphism Groups of Algebraic Varieties}, Math. Annalen 156 (1964), 25-33. 

\bibitem{Ser}
J.-P. Serre, \emph{Arbres, amalgames, $\mathrm{SL}_2$}. Ast\'erisque, No. {\bf 46.} Soci\'et\'e Math\'ematique de France, Paris, 1977.

\bibitem{bib:Sum}
H. Sumihiro, {\it Equivariant completion.} J. Math. Kyoto Univ. {\bf 14} (1974), 1--28.

\end{thebibliography}

\end{document}